\documentclass[11pt]{amsart}
\usepackage{latexsym}
\usepackage{graphicx,subfigure}
\usepackage{array}
\usepackage{indentfirst}
\usepackage{fancyhdr}
\usepackage{epsf}
\usepackage{amsmath,bm}
\usepackage{amsthm}
\usepackage{booktabs}
\usepackage{amsfonts}
\usepackage{epsf}
\usepackage{multicol}
\usepackage[square, comma, sort&compress, numbers]{natbib}
\usepackage{fancyhdr}
\usepackage{pxfonts}
\usepackage{caption}
\usepackage{epstopdf}
\usepackage{algorithm}
\usepackage{algpseudocode}
\usepackage{exscale}
\usepackage{relsize}
\usepackage{setspace}
\usepackage{geometry}
\usepackage{tikz}
\usetikzlibrary{arrows,shapes,chains,positioning}
\usetikzlibrary{decorations.markings}
\tikzstyle arrowstyle=[scale=1]
\tikzstyle directed=[postaction={decorate,decoration={markings,mark=at position .65 with {\arrow[arrowstyle]{stealth}}}}]
\tikzstyle reverse directed=[postaction={decorate,decoration={markings,mark=at position .65 with {\arrowreversed[arrowstyle]{stealth};}}}]

\newtheorem{Th}{Theorem}[section]
\newtheorem{Lm}{Lemma}[section]

\newtheorem{theorem}{Theorem}[section]
\newtheorem{lemma}[theorem]{Lemma}

\newtheorem{Ass}{Assumption}[section]
\theoremstyle{definition}
\numberwithin{equation}{section}
\newcommand{\bc}{\begin{center}}
\newcommand{\ec}{\end{center}}
\newcommand{\be}{\begin{eqnarray}}
\newcommand{\ee}{\end{eqnarray}}

\newcommand{\ben}{\begin{eqnarray*}}
\newcommand{\een}{\end{eqnarray*}}

\newcommand{\Om}{{\rm\Omega}}

\newcommand{\dx}{\,dx}

\newcommand{\ds}{\,ds}

\newcommand{\Rmnum}[1]{\expandafter\@slowromancap\romannumeral #1@}

\newcommand{\PiRT}{\Pi_{\rm RT}}
\newcommand{\PiCR}{\Pi_{\rm CR}}
\newcommand{\PiECR}{\Pi_{\rm ECR}}
\newcommand{\PiM}{\Pi_{\rm M}}
\newcommand{\PiPt}{\Pi_{\rm P_3}}
\newcommand{\VCR}{V_{\rm CR}}
\newcommand{\VECR}{V_{\rm ECR}}
\newcommand{\CRlam}{\lambda_{\rm CR}}
\newcommand{\ECRlam}{\lambda_{\rm ECR}}
\newcommand{\CRu}{u_{\rm CR}}
\newcommand{\ECRu}{u_{\rm ECR}}

\newcommand{\ECRuS}{u_{\rm ECR}^{\lambda \Pi_h^0 u}}
\newcommand{\ECRuSt}{u_{\rm ECR}^{\lambda u}}
\newcommand{\CRuSt}{u_{\rm CR}^{\lambda u}}
\newcommand{\CRuS}{u_{\rm CR}^{\lambda u}}
\newcommand{\RTuS}{u_{\rm RT}^{\lambda u}}
\newcommand{\RTsigS}{\sigma_{\rm RT}^{\lambda u}}
\newcommand{\CRuPi}{u_{\rm CR}^{\lambda \Pi_h^0 u}}

\newcommand{\RTbeta}{\gamma_{\rm RT}}

\newcommand{\cE}{\mathcal{E}}

\newcommand{\cO}{\mathcal{O}}

\newcommand{\cT}{\mathcal{T}}
\newcommand{\R}{\mathbb{R}}

\newcommand{\bp}{\boldsymbol{p}}
\newcommand{\bM}{\boldsymbol{m}}

\renewcommand{\arraystretch}{1.5}

\newcommand{\yemeifont}{\fontsize{9pt}{\baselineskip}\selectfont}
\begin{document}
\title{
Asymptotic expansions of eigenvalues by both the Crouzeix--Raviart and enriched Crouzeix--Raviart elements
}

\author {Jun Hu}
\address{LMAM and School of Mathematical Sciences, Peking University,
  Beijing 100871, P. R. China.  hujun@math.pku.edu.cn}

\author{Limin Ma}
\address{LMAM and School of Mathematical Sciences, Peking University,
  Beijing 100871, P. R. China. maliminpku@gmail.com}

\thanks{The authors were supported by  NSFC
projects 11625101 and 11421101}

\maketitle

\begin{abstract}
Asymptotic expansions  are derived for eigenvalues produced by both the Crouzeix-Raviart element and the enriched Crouzeix--Raviart element. The expansions are optimal in the sense that extrapolation eigenvalues based on them admit a fourth order convergence provided that exact eigenfunctions are smooth enough. The major challenge in establishing the expansions comes from the fact that the canonical interpolation of both nonconforming elements  lacks a crucial superclose property, and  the nonconformity of both elements.
The main idea is to employ the relation between the lowest-order mixed Raviart--Thomas element and the two nonconforming elements, and consequently make use of the superclose property of the canonical  interpolation of the lowest-order mixed Raviart--Thomas element.
To overcome the difficulty caused by  the nonconformity, the commuting property of the canonical interpolation operators of both nonconforming elements is further used, which turns the consistency error problem into an interpolation error problem. Then, a series of new results are obtained to show the final expansions.

  \vskip 15pt


 \vskip 15pt


\end{abstract}

\section{Introduction}
Asymptotic expansions of approximate solutions guarantee the efficiency of extrapolation methods.
The classical analysis of asymptotic expansions is usually carried out
by using the superclose property of the canonical interpolation of the element under consideration, see for instance
 \cite{Lin1984Asymptotic, Ding1990quadrature,Lin2011Extrapolation,Blum1990Finite,Lin2008New,Lin2009Asymptotic,lin1999high,Chen2007Asymptotic,
Jia2010Approximation,Luo2002High,lin2010new,Lin2007Finite,Lin2009New} and the references therein. For the Crouzeix-Raviart (CR for short hereinafter) element, the extrapolation methods in \cite{Lin2005CAN} were examined to improve the accuracy of discrete eigenvalues from second order to fourth order numerically.  But no asymptotic expansions were  analyzed there to justify the experimental results. One major difficulty comes from the fact  that the canonical interpolation of  the CR element does not admit such a superclose property.

In this paper, asymptotic expansions of eigenvalues on uniform triangulations are explored for both the CR element and the  enriched Crouzeix-Raviart (ECR for short hereinafter) element   for the first time. Errors of eigenvalues by nonconforming elements admit the following identity in \cite{hu2019posteig}
\begin{equation*}
\begin{split}
\lambda -\lambda_h =\|\nabla_h (u -  u_h)\|_{0, \Om}^2+2 a_h(u ,u_h )-2\lambda_h (u , u_h )-\lambda_h \|u - u_h\|_{0, \Om}^2
\end{split}
\end{equation*}
with approximate eigenpairs $(\lambda_h, u_h)$ defined in \eqref{discrete} below. Compared to conforming elements, there exist two major difficulties.  The canonical interpolation of the nonconforming elements does not admit a superclose property. This leads to the difficulty in expanding the first term $\|\nabla_h (u -  u_h)\|_{0, \Om}^2$ with high accuracy.
The nonconformity causes the other difficulty in expanding the  consistency error term $a_h(u ,u_h ) - \lambda_h (u , u_h )$ for nonconforming elements.

 One major idea to overcome the first difficulty is to employ  the relation between the lowest-order mixed Raviart--Thomas element (RT for short hereinafter) and  both the CR element and  the ECR element, and  to exploit the superconvergence result \cite{hu2018optimal} for the  RT element. This superconvergence property of the mixed element remedies the lack of the superclose property of the canonical interpolation of both nonconforming elements. To overcome the second difficulty, the main idea is to make use of the commuting property of the canonical interpolation operator of both nonconforming elements. This commuting property turns the consistency error term into an interpolation error term. Take the CR element as an example,
it follows from the aforementioned superconvergence of the  RT element, the commuting property of the CR element and the special relation between the CR element and the RT element \cite{Marini1985An} that
\begin{equation}\label{eq:intro}
\lambda-\CRlam=\parallel (I-\PiRT)\nabla u \parallel_{0,\Om}^2   + \frac{\lambda^2  H^2}{144}+I_{\rm CR} +I_{\rm RT}+I_{\rm CR}^1+I_{\rm CR}^2+\cO(h^4|\ln h||u|_{{7\over 2},\Om}^2),
\end{equation}
with $H$ defined in  \eqref{HK}, the  terms $I_{\rm CR}$, $I_{\rm RT}$, $I_{\rm CR}^1$  and $I_{\rm CR}^2$ defined in  \eqref{crI}.
Optimal expansions of eigenvalues require fourth-order accurate expansions of $\parallel (I-\PiRT)\nabla u \parallel_{0,\Om}^2$ and $I_{\rm CR}$ and also an optimal analysis of the other terms in \eqref{eq:intro}.

There are three key terms for the expansions. The first one is $\parallel (I-\PiRT)\nabla u \parallel_{0,\Om}^2$, whose
optimal expansion  needs to introduce an operator with some commuting property, and a refined analysis of the associated interpolation error. The second one is  $I_{\rm CR}$ which contains two terms, one is  essentially a consistency error and  only admits a third order convergence which can not be improved. Hence, the direct use of the Cauchy-Schwarz inequality and the Taylor expansions of interpolation errors only leads to a suboptimal expansion. The idea is to  decompose the first term of $I_{\rm CR}$ into two terms: one cancels this consistency error, and the other term has an asymptotic expansion. One key result for the analysis is a crucial superconvergence of the   inner product of the errors of the  canonical interpolation  of the CR element and the piecewise constant $L^2$ projection.
The third term is $I_{\rm CR}^1$. For it, a direct combination of Cauchy-Schwarz inequality and the superclose property  of the CR element only yields  a suboptimal estimate.  The idea here is to  make use of  the relation  between the CR element and the RT element
 and decompose it into three terms: a vanishing term, a fourth order term and a remaining term. By using the commuting property of 
 the canonical interpolation operator of the CR element, the discrete eigenvalue problem  and  an auxiliary  discrete source problem, and fully  exploring the properties of the piecewise constant $L^2$ projection operator and the uniformity  of the mesh, this remaining term can be in some sense transferred to a  consistency error. A key result is the superconvergence of the inner 
  product of the errors of the piecewise constant $L^2$ projections of two CR element functions.

The remaining paper is organized as follows. Section 2 presents second order elliptic eigenvalue problems and some notations. Section 3 explores optimal asymptotic expansions of approximate eigenvalues of the CR element and analyzes the optimal convergence rate of eigenvalues by extrapolation methods. Section 4 deals with  eigenvalues of the ECR element in a similar way to that of the CR element in Section 3. Section 5 presents some numerical tests.

\section{Notations and Preliminaries}
\subsection{Notations}\label{sec:notation}
Given a nonnegative integer $k$ and a bounded domain $\Om\subset \mathbb{R}^2$ with boundary $\partial \Om$, let $W^{k,\infty}(\Om,\mathbb{R})$, $H^k(\Om,\mathbb{R})$, $\parallel \cdot \parallel_{k,\Om}$ and $|\cdot |_{k,\Om}$ denote the usual Sobolev spaces, norm, and semi-norm, respectively. And $H_0^1(\Om,\mathbb{R}) = \{u\in H^1(\Om,\mathbb{R}): u|_{\partial \Om}=0\}$. Denote the standard $L^2(\Om,\mathbb{R})$ inner product and $L^2(K,\mathbb{R})$ inner product by $(\cdot, \cdot)$ and $(\cdot, \cdot)_{0,K}$, respectively.

Suppose that $\Om\subset \mathbb{R}^2$ is a bounded polygonal domain covered exactly by a shape-regular partition $\cT_h$ into simplices. Let $|K|$ denote the area of element $K$ and $|e|$ the length of edge $e$. Let $h_K$ denote the diameter of element $K\in \cT_h$ and $h=\max_{K\in\cT_h}h_K$. Denote the set of all interior edges and boundary edges of $\cT_h$ by $\cE_h^i$ and $\cE_h^b$, respectively, and $\cE_h=\cE_h^i\cup \cE_h^b$. For any interior edge $e=K_e^1\cap K_e^2$, denote the element with larger global label by $K_e^1$ and the one with smaller global label by $K_e^2$. Denote the corresponding unit normal vector which points from $ K_e^1 $ to $K_e^2$ by $\bold{n}_e$. Let $[\cdot]$ be the jump of piecewise functions over edge $e$, namely
$$
[v]|_e := v|_{K_e^1}-v|_{K_e^2}$$
for any piecewise function $v$. For $K\subset\R^2,\ r\in \mathbb{Z}^+$, let $P_r(K, \R)$ be the space of all polynomials of degree not greater than $r$ on $K$. For $r\geq 1$, denote
$$
\nabla P_r(K, \R^2):=\{\nabla v: v\in P_r(K, \R)\}.
$$
Denote the piecewise gradient operator and the piecewise hessian operator by $\nabla_h$ and $\nabla_h^2$, respectively.

Let element $K$ have vertices $\bold{p}_i=(p_{i1},p_{i2}),1\leq i\leq 3$ oriented counterclockwise, and corresponding barycentric coordinates $\{\psi_i\}_{i=1}^3$. Let $\{e_i\}_{i=1}^3$  denote the edges of element $K$, $\{d_i\}_{i=1}^3$ the perpendicular heights,  $\{\theta_i\}_{i=1}^3$ the internal angles, $\{\bold{m}_i\}_{i=1}^3$ the  midpoint of edge $\{e_i\}_{i=1}^3$, and $\{\bold{n}_i\}_{i=1}^3$ the unit outward normal vectors, $\{\bold{t}_i\}_{i=1}^3$ the unit tangent vectors with counterclockwise orientation (see Figure \ref{fig:geometric}). There holds the following relationships
$
d_i|e_i|=2|K|
$
and
\be\label{nlambda}
\nabla \psi_i=-\frac{\bold{n}_i}{d_i}.
\ee
among the quantities \cite{huang2008superconvergence}.
Denote
\begin{equation}\label{HK}
H_K^2=\sum_{i=1}^3|e_i|^2 \quad (\text{ on uniform meshes, $H_K$ will be denoted by $H$})
\end{equation}
and the centroid of element $K$ by $\bold{M}_K=(M_1, M_2)$. Denote the second order derivatives $\frac{\partial^2 u}{\partial x_i\partial x_j}$ by $\partial_{x_ix_j} u$, $1\leq i, j\leq 2$.
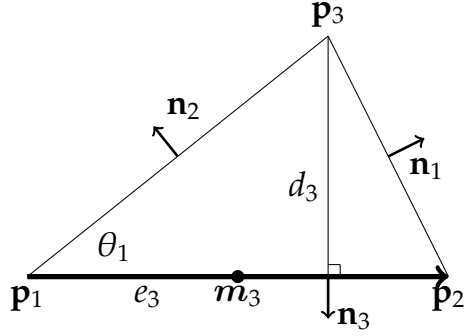
\begin{figure}[!ht]
\begin{center}
\begin{tikzpicture}[xscale=8,yscale=8]
\tikzstyle{every node}=[font=\Large,scale=0.9]
\draw[->][line width=2pt](0,0) -- (0.7,0);
\draw[-] (0,0) -- (0.5,0.4);
\draw[-] (0.7,0) -- (0.5,0.4);
\draw[-] (0.5,0) -- (0.5,0.4);
\draw[-] (0.5,0.02) -- (0.52,0.02);
\draw[-] (0.52,0.02) -- (0.52,0);
\draw[->][line width=1pt] (0.5,0) -- (0.5,-0.07);
\draw[->][line width=1pt] (0.6,0.2) -- (0.66,0.23);
\draw[->][line width=1pt] (0.25,0.2) -- (0.21,0.25);
\node[below] at (0,0) {$\bold{p}_1$};
\node[below] at (0.7,0) {$\bold{p}_2$};
\node[above] at (0.5,0.4) {$\bold{p}_3$};
\node[below] at (0.35,0) {$\bM_3$};
\draw [fill] (0.35,0) circle [radius=0.01];
\node[below] at (0.2,0) {$e_3$};
\node[right] at (0.5,-0.07) {$\bold{n}_3$};
\node[right] at (0.62,0.18) {$\bold{n}_1$};
\node at (0.26,0.28) {$\bold{n}_2$};
\node[left] at (0.5,0.15) {$d_3$};
\node[right] at (0.1,0.05) {$\theta_1$};
\end{tikzpicture}
\end{center}
\caption{Paramters associated with a triangle $K$.}
\label{fig:geometric}
\end{figure}
For ease of presentation,  the symbol $A\lesssim B$ will be used to denote that $A\leq CB$, where $C $ is a positive constant.
\subsection{Nonconforming elements for eigenvalue problems}\label{sec:model}
Consider a model eigenvalue problem of finding : $(\lambda, u)\in \mathbb{R}\times V$  such that  $\parallel u \parallel_{0,\Om}=1$ and
\be\label{variance}
a(u, v)=\lambda(u, v) \text{\quad for any }v\in V,
\ee
with $V:= H^1_0(\Om,\mathbb{R})$. The bilinear form
$
a(w, v):=\int_{\Om} \nabla  w\cdot \nabla v \dx
$
is symmetric, bounded, and coercive, namely for any $w, v\in V$,
$$
a(w,v)=a(v,w),\quad |a(w,v)|\lesssim \parallel w\parallel_{1,\Om}\parallel v\parallel_{1,\Om},\quad \parallel v\parallel_{1,\Om}^2\lesssim a(v,v).
$$
The eigenvalue problem \eqref{variance} has a sequence of eigenvalues
$$0<\lambda^1\leq \lambda^2\leq \lambda^3\leq ...\nearrow +\infty,$$
and the corresponding eigenfunctions
$u^1, u^2, u^3,... ,$
with
$$(u^i, u^j)=\delta_{ij}\ \text{ with } \delta_{ij}=\begin{cases}
0 &\quad i\neq j\\
1 &\quad i=j
\end{cases}.
$$

Let $V_h$ be a nonconforming finite element approximation of $V$ over $\cT_h$. The corresponding finite element approximation of \eqref{variance} is to find $(\lambda_h, u_h)\in \mathbb{R}\times V_h$  such that $\parallel u_h\parallel_{0,\Om}=1$ and
\be\label{discrete}
a_h(u_h,v_h)=\lambda_h(u_h, v_h)\quad \text{ for any }v_h\in V_h,
\ee
with the discrete bilinear form
$
a_h(w_h,v_h):=\sum_{K\in\cT_h}\int_K \nabla_h w_h\cdot \nabla_h v_h\dx.
$


Consider the following two nonconforming elements: the CR element and the ECR element.

$\bullet$\quad  The CR element space over $\cT_h$ is defined in \cite{Crouzeix1973Conforming} by
\begin{equation*}
\begin{split}
\VCR:=&\big \{v\in L^2(\Om,\R)\big|v|_K\in P_1(K, \R)\text{ for any }  K\in\cT_h, \int_e [v]\ds =0\text{ for any }  e\in \cE_h^i,\\
&\int_e v\ds=0\text{ for any }  e\in \cE_h^b\big\}.
\end{split}
\end{equation*}
The corresponding canonical interpolation operator $\PiCR: V\rightarrow \VCR$ is defined as follows:
\be\label{crinterpolation}
\int_e\PiCR v\ds=\int_e v\ds\quad \text{ for any } e\in \cE_h,\ v\in V.
\ee
Denote the approximate eigenpair of \eqref{discrete} with $V_h=\VCR$ by $(\CRlam,\CRu)$ with $ \parallel \CRu\parallel_{0,\Om}=1$.

$\bullet$\quad  The ECR element space over $\cT_h$ is defined in \cite{Hu2014Lower} by
\begin{equation*}
\begin{split}
\VECR:=&\big \{v\in L^2(\Om,\mathbb{R})\big|v|_K\in \rm{ECR(K, \R)}\text{ for any }  K\in\cT_h, \int_e [v]\ds =0\text{ for any }  e\in \cE_h^i,\\
&\int_e v\ds=0\text{ for any }  e\in \cE_h^b\big\}.
\end{split}
\end{equation*}
with $\rm{ECR(K,\R)}: =P_1(K, \R)+\text{span}\big\{x_1^2+x_2^2\big\}$. The corresponding canonical interpolation operator $\PiECR: V\rightarrow \VECR $ is defined by
\begin{equation}\label{ecrinterpolation}
\int_e\PiECR v\ds= \mathlarger{\int}_e v\ds,\quad \mathlarger{\int}_K \PiECR v\dx= \int_K v\dx\quad\text{ for any }  e\in\cE_h, K\in\cT_h.
\end{equation}
Denote the approximate eigenpair of \eqref{discrete}  $V_h=\VECR$ by $(\ECRlam,\ECRu)$ and $ \parallel \ECRu\parallel_{0,\Om}=1$.

\begin{Ass}\label{ass:regularity}
The domain $\Om$ is convex or eigenfunction $u$ is smooth.
\end{Ass}
Assumption \ref{ass:regularity} guarantees that eigenfunctions belong to $H^2(\Om, \R)$. It follows from the theory of nonconforming eigenvalue approximations, see for instance, \cite{Hu2014Lower,rannacher1979nonconforming} and the references therein,  that
\begin{equation}\label{CR:est}
|\lambda-\CRlam|+\parallel u- \CRu\parallel_{0,\Om} +\parallel u- \PiCR u\parallel_{0,\Om}+ h\parallel \nabla_h (u-\CRu)\parallel_{0,\Om}\lesssim h^{2}\parallel u\parallel_{2,\Om},
\end{equation}
\begin{equation}\label{ECR:est}
|\lambda-\ECRlam|+\parallel u - \ECRu\parallel_{0,\Om} +\parallel u- \PiECR u\parallel_{0,\Om}+ h\parallel \nabla_h (u-\ECRu)\parallel_{0,\Om}\lesssim h^{2}\parallel u\parallel_{2,\Om}.
\end{equation}
For the CR element and the ECR element, there holds the following commuting property for their canonical interpolations
\begin{equation}\label{commuting}
\begin{split}
\int_K \nabla(w - \PiCR w)\cdot \nabla v_h \dx &=0\quad\text{ for any } w\in V, v_h\in \VCR,\\
\int_K \nabla(w - \PiECR w)\cdot \nabla v_h \dx &=0\quad\text{ for any }w\in V, v_h\in \VECR,
\end{split}
\end{equation}
see \cite{Crouzeix1973Conforming,Hu2014Lower} for more details.

For the CR element, there exists the following identity for the error of the approximate eigenvalues \cite{hu2019posteig}
\begin{equation}\label{originalId}
\begin{split}
\lambda -\CRlam =\|\nabla_h (u - \CRu)\|_{0, \Om}^2+2 a_h(u ,\CRu )-2\CRlam (u , \CRu )-\CRlam \|u - \CRu\|_{0, \Om}^2.
\end{split}
\end{equation}
It is difficult to establish an asymptotic expansion for the consistency error term $a_h(u ,\CRu )-\CRlam (u , \CRu )$ directly. The main idea herein is to employ the canonical interpolation operator $\PiCR$ of the CR element and the crucial commuting property \eqref{commuting}. In this way,  the consistency error term can be expressed in terms of the interpolation error, namely
$$
a_h(u ,\CRu )-\CRlam (u , \CRu )=-\CRlam (u-\PiCR u,\CRu ).
$$
As a result, the identity \eqref{originalId} becomes
\begin{equation}\label{commutId}
\lambda-\CRlam =\|\nabla_h (u - \CRu)\|_{0, \Om}^2-2\CRlam (u-\PiCR u,\CRu )-\CRlam \|u - \CRu\|_{0, \Om}^2.
\end{equation}
The asymptotic expansions of eigenvalues of the CR element in this paper are based on this crucial identity \eqref{commutId}.
Since the ECR element also admits a commuting property, a similar identity to \eqref{commutId} holds for approximate eigenpairs $(\ECRlam , \ECRu )$, and leads to the asymptotic expansions of eigenvalues of the ECR element.

\subsection{Raviart--Thomas element for source problems}

The shape function space of the lowest order RT element \cite{Raviart1977A}  is as follows
$$
\rm{RT(K, \R^2)}:=P_0(K, \R^2)+ \bold{x}P_0(K, \R)\quad\text{ for any }K\in \cT_h.
$$
The corresponding  finite element space reads
$$
\text{RT}(\cT_h):=\big \{\tau\in H(\text{div},\Om,\mathbb{R}^2): \tau|_K\in \rm{RT(K, \R^2)}\text{ for any }K\in \cT_h\big \}.
$$
To get a stable pair of space, the piecewise constant space is used to approximate the displacement, namely,
$$
U_{\text{RT}}:=\big \{v\in L^2(\Om, \R):v|_K\in P_0(K, \R) \text{ for any }K\in \cT_h\big \}.
$$
The Fortin interpolation operator $\PiRT:H(\rm{div}, \Om,\mathbb{R}^2)\rightarrow \text{RT}(\cT_h)$, which is widely used in error analysis, see for instance \cite{Douglas1985Global,Dur1990Superconvergence} and the references therein, is defined by
\begin{equation}\label{def:fortin}
\int_e (\PiRT \bold{\tau}-\bold{\tau})^T\bold{n}_e\ds=0\text{\quad for any }e\in \cE_h, \bold{\tau}\in H(\rm{div}, \Om,\R^2).
\end{equation}
For the analysis, introduce the following discrete source problem which seeks $(\RTsigS,\RTuS) \in \text{RT}(\cT_h)\times U_{\text{RT}}$ such that
\begin{equation}\label{RTbdPro}
\begin{aligned}
(\RTsigS,\tau_h)+(\RTuS,\text{div}\tau_h)&=0&& \text{ for any }\tau_h\in \text{RT}(\cT_h),\\
(\text{div}\RTsigS,v_h)&=-\lambda (u,v_h)&&\text{ for any }v_h\in U_{\text{RT}}.
\end{aligned}
\end{equation}
Note that  $\RTsigS$ is the RT element solution of $\sigma^{\rm \lambda u}=\nabla u$. It follows from the theory of mixed finite element methods \cite{Douglas1985Global} that
\be\label{RT:est}
\parallel u- \RTuS\parallel_{0,\Om} + \parallel \nabla u -\RTsigS\parallel_{0,\Om}+ \parallel \rm{div} (\sigma^{\rm \lambda u} -\RTsigS)\parallel_{0,\Om}\lesssim h^{s}\parallel u\parallel_{1+s,\Om},
\ee
provided that $u\in H^{1+s}(\Om,\mathbb{R})\cap H^1_0(\Om,\mathbb{R})$, $\ 0< s\leq 1$.

In this paper, the family of triangulations is assumed to be uniform:
\begin{Ass}\label{ass:mesh}
The triangulation $\cT_h$ is uniform. This is, any two adjacent triangles form a parallelogram.
\end{Ass}

In this paper, only uniform meshes in this sense will be used.  For this case, the mesh dependent constant $H_K$
from \eqref{HK} will be denoted by $H$.

According to \cite{Brandts1994Superconvergence,li2017global,hu2018optimal}, the RT element admits an important superconvergence property on uniform triangulations.
\begin{lemma}\label{Lm:bdRT}
Suppose that $(\RTsigS,\RTuS)$ is the solution of problem \eqref{RTbdPro} and $u\in H^{\frac{7}{2}}(\Om,\mathbb{R})\cap H^1_0(\Om, \R)$. Under the Assumption \ref{ass:mesh},
\begin{equation}\label{ECR1}
\parallel \RTsigS-\PiRT\nabla u\parallel_{0,\Om}\lesssim h^2 \big (| u|_{\frac{7}{2},\Om}+\kappa |\ln h|^{1/2}|u|_{2,\infty,\Om}\big ).
\end{equation}
\end{lemma}
\subsection{Relation between the RT element and nonconforming elements}
Define the $L^2$ projection operators $\Pi_K^0: L^2(K, \R)\rightarrow P_0(K, \R)$ and $\Pi_h^0: L^2(\Om, \R)\rightarrow U_{\rm RT}$ by
\be\label{Pi0}
\Pi_K^0 w =\frac{1}{|K|}\int_K w\dx\quad \text{ and }\quad \Pi_h^0 w\big |_K =\Pi_K^0 u,
\ee
respectively. It holds that for any $u\in H^1(K, \R)$,
\be\label{Pi0est}
\|u- \Pi_K^0 u\|_{0, K}\lesssim h|u|_{1, K}.
\ee

Consider  two discrete source problems: one seeks $\CRuPi\in \VCR$ such that
\be\label{CRbdPipro}
(\nabla_h \CRuPi,\nabla_h v_h)=\lambda (\Pi_h^0 u ,v_h)\quad \text{ for any } v_h\in \VCR,
\ee
and the other one seeks $\ECRuS \in \VECR$ such that
\be\label{ECRbdpro}
(\nabla_h \ECRuS ,\nabla_h v_h)=(\lambda \Pi_h^0 u ,v_h)\quad\text{ for any } v_h\in \VECR.
\ee
By the definition of the $L^2$ projection operator $\Pi_h^0$ and the discrete space $U_{\text{RT}}$,  it holds that
$$
(\lambda u, v_h)= (\lambda \Pi_h^0 u, v_h), \text{ for any } v_h\in U_{\text{RT}} .
$$
This implies that the discrete source problem \eqref{RTbdPro}  can be  rewritten as: finding $(\RTsigS,\RTuS) \in \text{RT}(\cT_h)\times U_{\text{RT}}$ such that
\begin{equation*}
\begin{aligned}
(\RTsigS,\tau_h)+(\RTuS,\text{div}\tau_h)&=0&& \text{ for any }\tau_h\in \text{RT}(\cT_h),\\
(\text{div}\RTsigS,v_h)&=-\lambda (\Pi_h^0 u, v_h)&&\text{ for any }v_h\in U_{\text{RT}}.
\end{aligned}
\end{equation*}

The following Lemma states the special relations between the RT solution of Problem \eqref{RTbdPro} and  the CR solution of Problem \eqref{CRbdPipro},  and the ECR solution of Problem \eqref{ECRbdpro}, respectively,  see more details in  \cite{Marini1985An} and  \cite{Hu2015The}.
\begin{lemma}\label{lm:equiv}
Let $(\RTsigS, \RTuS)$, $\CRuPi$ and $\ECRuS$ be the solutions of \eqref{RTbdPro}, \eqref{CRbdPipro} and \eqref{ECRbdpro}, respectively. It holds that
\begin{equation}\label{CRRT}
\RTsigS\big|_K=\nabla_h \CRuPi\big|_K-\frac{\lambda  \Pi_K^0 u }{2}(\bold{x}-\bold{M}_K),\quad \bold{x}\in K,
\end{equation}
\begin{equation}\label{ECRRT}
\RTsigS= \nabla_h \ECRuS,
\end{equation}
for any $K\in\cT_h$ with $\bold{M}_K$ the centroid of $K$.
\end{lemma}

\section{Asymptotic expansions of eigenvalues of the CR element}\label{sec:CR}

In this section, asymptotic expansions of eigenvalues are established for the CR element, and then employed to prove an optimal convergence of eigenvalues by extrapolation methods.

\subsection{Error expansions for eigenvalues}
Let $\CRuS \in V_{\rm CR}$ be the solution of the following source problem
\be\label{CRbdpro}
(\nabla_h \CRuS , \nabla_h v_h)=\lambda (u,v_h)\quad \text{ for any } v_h\in \VCR,
\ee
where $(\lambda, u)$ is the solution of the eigenvalue problem \eqref{variance}. The following lemma presents a superclose property of the discrete eigenfunction $\CRu$ with respect to the CR solution $\CRuPi$ of the discrete source problem  \eqref{CRbdPipro}.
\begin{lemma}\label{supereig1}
Suppose that $(\lambda, u)$ is the eigenpair  of \eqref{variance}, $(\CRlam, \CRu)$ is the corresponding  eigenpair of  \eqref{discrete} by the CR element and $\CRuPi$ is the  solution of  \eqref{CRbdPipro}. Under the Assumption \ref{ass:regularity},
\begin{equation}\label{CRESdiff}
\parallel\nabla_h \CRu-\nabla_h \CRuPi \parallel_{0,\Om}\lesssim h^2|u|_{2,\Om}.
\end{equation}
\end{lemma}
\begin{proof}
It follows from the triangle inequality that
\begin{equation}\label{CRESdiffT}
\parallel\nabla_h \CRu-\nabla_h \CRuPi \parallel_{0,\Om}\leq \parallel\nabla_h \CRu-\nabla_h \CRuSt \parallel_{0,\Om} + \parallel\nabla_h \CRuSt-\nabla_h \CRuPi \parallel_{0,\Om}.
\end{equation}
To bound the first term on the right--hand side of \eqref{CRESdiffT}, let $v_h=\CRu-\CRuSt$ in Problem \eqref{discrete} and Problem \eqref{CRbdpro}.  This  yields
\begin{equation}\label{CRESdiff0}
\parallel\nabla_h \CRu-\nabla_h \CRuSt \parallel_{0,\Om}^2=(\CRlam \CRu - \lambda u, \CRu-\CRuSt ).
\end{equation}
By the error estimate \eqref{CR:est} for the discrete eigenpair $( \CRlam, \CRu)$,
\begin{equation}\label{CRESdiff1}
\parallel \CRlam \CRu - \lambda u\parallel_{0,\Om}\leq |\CRlam|\parallel \CRu - u\parallel_{0,\Om}+ | \CRlam -  \lambda |\lesssim h^2|u|_{2,\Om}.
\end{equation}
It follows from the triangle inequality, the error estimate  \eqref{CR:est}, and the fact that $\CRuSt$ is the CR element approximation
 of the solution $u$  of the source problem $-\Delta w=\lambda u$ with $(\lambda, u)$ the eigenpair of Problem \eqref{variance},   that
\begin{equation}\label{CRESdiff2}
\parallel \CRu-\CRuSt \parallel_{0,\Om}\leq \parallel \CRu-u\parallel_{0,\Om}+ \parallel u-\CRuSt \parallel_{0,\Om}\lesssim h^2|u|_{2,\Om}.
\end{equation}
A substitution of \eqref{CRESdiff1} and \eqref{CRESdiff2} to \eqref{CRESdiff0} leads to
\begin{equation}\label{CRESdiff3}
\parallel\nabla_h \CRu-\nabla_h \CRuSt \parallel_{0,\Om}\lesssim h^2|u|_{2,\Om}.
\end{equation}
To bound the second term on the right--hand side of \eqref{CRESdiffT}, let $v_h=\CRu-\CRuSt$ in Problem  \eqref{CRbdPipro} and Problem \eqref{CRbdpro}.  This  yields
\begin{equation*}
\parallel\nabla_h \CRuSt-\nabla_h \CRuPi \parallel_{0,\Om}^2=\lambda ( u- \Pi_h^0 u, \CRuSt - \CRuPi )=\lambda ( u- \Pi_h^0 u, (I-\Pi_h^0)(\CRuSt - \CRuPi) ).
\end{equation*}
This, together with the error estimate  \eqref{Pi0est} of the piecewise constant $L^2$-projection $\Pi_h^0 u$ and a Poincare inequality
 for the term $(I-\Pi_h^0)(\CRuSt - \CRuPi)$, leads to
$$
\parallel\nabla_h \CRuSt-\nabla_h \CRuPi \parallel_{0,\Om}^2\lesssim \lambda h^2|u|_{1, \Om}\parallel\nabla_h \CRuSt-\nabla_h \CRuPi \parallel_{0,\Om},
$$
consequently,
\begin{equation}\label{CRESdiff5}
\parallel\nabla_h \CRuSt-\nabla_h \CRuPi\parallel_{0,\Om}\lesssim \lambda h^2|u|_{1, \Om}.
\end{equation}
A substitution of \eqref{CRESdiff3} and \eqref{CRESdiff5} to \eqref{CRESdiffT} leads to
$$
\parallel\nabla_h \CRu-\nabla_h \CRuPi \parallel_{0,\Om}\lesssim h^2|u|_{2,\Om},
$$
and completes the proof.
\end{proof}

For the CR element, the canonical interpolation does not admit the usual superclose property with respect to the finite element solution in the energy norm. The lack of this crucial superclose property makes it difficult to establish asymptotic expansions of eigenvalues by directly using the canonical interpolation of the nonconforming CR element.

To overcome such a difficulty, the key idea is to make use of  the relation between the RT element and the CR element in Lemma \ref{lm:equiv} and the superconvergence property of the mixed RT element. To this end, introduce the following notations
with
\begin{equation}\label{crI}
\begin{array}{ll}
I_{\rm CR}&=2((I-\PiRT)\nabla u   , \RTsigS-\nabla_h \CRuPi)-2\lambda(u-\PiCR u,  u),\\
I_{\rm RT}&=2(\nabla u-\PiRT\nabla u, \PiRT\nabla u-\RTsigS), \\
 I_{\rm CR}^1&=2(\nabla u -\PiRT \nabla u , \nabla_h \CRuPi-\nabla_h \CRu),\\
I_{\rm CR}^2&=2(\PiRT\nabla u -\RTsigS,\RTsigS- \nabla_h \CRu).
\end{array}
\end{equation}
The asymptotic  expansions of eigenvalues of the CR element are based on the identity in the following theorem.
\begin{Th}\label{Th:extraCR}
Suppose that $(\lambda , u )$ is the eigenpair of \eqref{variance} with $u \in H^{\frac{7}{2}}(\Om,\mathbb{R})\cap H^1_0(\Om,\mathbb{R})$, and $(\CRlam, \CRu)$ is the discrete eigenpair  of \eqref{discrete} by the CR element. Under  Assumption \ref{ass:mesh},
\begin{equation}\label{CRfinal1}
\begin{split}
\lambda-\CRlam&=\parallel (I-\PiRT)\nabla u \parallel_{0,\Om}^2   + \frac{\lambda^2  H^2}{144}+I_{\rm CR} +I_{\rm RT}+I_{\rm CR}^1+I_{\rm CR}^2+\cO(h^4|\ln h||u|_{{7\over 2},\Om}^2),
\end{split}
\end{equation}
with  $H$ defined in \eqref{HK}, and $I_{\rm CR}$ , $I_{\rm RT}$, $I_{\rm CR}^1$ and $I_{\rm CR}^2$  defined in \eqref{crI}.
\end{Th}
\begin{proof}
Recall the expansion \eqref{commutId} of eigenvalues by the CR element as follows
$$
\lambda-\CRlam =\|\nabla_h (u - \CRu)\|_{0, \Om}^2-2\CRlam (u-\PiCR u,\CRu )-\CRlam \|u - \CRu\|_{0, \Om}^2.
$$
The second term on the right-hand side  can be decomposed into three terms as follows
\begin{equation}\label{CRexp2}
\begin{split}
\CRlam(u-\PiCR u,\CRu)=&\lambda(u-\PiCR u, u) + (\CRlam-\lambda)(u-\PiCR u,u) \\
&+ \CRlam(u-\PiCR u,\CRu-u).
\end{split}
\end{equation}
By the error estimates \eqref{CR:est} of the CR element,
\begin{equation*}
\big |(\CRlam -\lambda)(u -\PiCR u,u)\big |\lesssim |\CRlam -\lambda| \parallel u-\PiCR u \parallel_{0,\Om}\lesssim h^4|u|_{2,\Om}^2,
\end{equation*}
\begin{equation*}
\big |(u-\PiCR u,\CRu-u)\big |\lesssim \parallel u-\PiCR u \parallel_{0,\Om}\parallel \CRu-u\parallel_{0,\Om}\lesssim h^4|u|_{2,\Om}^2.
\end{equation*}
A substitution of these two estimates into  \eqref{CRexp2} yields
\begin{equation}\label{CRlambdaexpand}
\CRlam(u-\PiCR u,\CRu)=\lambda(u-\PiCR u, u )+\cO(h^4|u|_{2,\Om}^2).
\end{equation}
Note that
$$
\nabla u-\nabla_h \CRu= (\nabla u -\PiRT\nabla u) + (\PiRT\nabla u -\RTsigS) + (\RTsigS- \nabla_h \CRuPi) + (\nabla_h \CRuPi- \nabla_h \CRu),
$$
with the RT solution $\RTsigS$ of the source problem \eqref{RTbdPro} and the CR solution $\CRuPi$ of the modified problem \eqref{CRbdPipro}.
A substitution of the above decomposition and  \eqref{CRlambdaexpand} into \eqref{commutId} leads to
\begin{equation}\label{CRtotal0}
\begin{split}
\lambda-\CRlam=&\parallel \nabla u -\PiRT\nabla u  \parallel_{0,\Om}^2 + \parallel \PiRT\nabla u -\RTsigS \parallel_{0,\Om}^2+ \parallel \RTsigS- \nabla_h \CRuPi\parallel_{0,\Om}^2 \\
&+\parallel \nabla_h \CRuPi- \nabla_h \CRu\parallel_{0,\Om}^2 +2(\RTsigS- \nabla_h \CRuPi, \nabla_h \CRuPi-\nabla_h \CRu)\\
&+I_{\rm CR}+I_{\rm RT}+ I_{\rm CR}^1+I_{\rm CR}^2+\cO(h^4|u|_{2,\Om}^2)
\end{split}
\end{equation}
with $I_{\rm CR}$ , $I_{\rm RT}$, $I_{\rm CR}^1$ and $I_{\rm CR}^2$  defined in \eqref{crI}.
By the definition of the $L^2$ projection operator $\Pi_K^0$ and the corresponding error estimate \eqref{Pi0est},
\begin{equation*}
\parallel \Pi_K^0 u \parallel_{0,K}^2=\|u\|_{0,K}^2 - \|u -\Pi_K^0 u\|_{0,K}^2 \quad \text{ with  }\quad \|u -\Pi_K^0 u\|_{0,K}\lesssim h|u|_{1, K}.
\end{equation*}
According to the relation \eqref{CRRT} between the CR element and the RT element, namely $2(\RTsigS- \nabla_h \CRuPi)|_K=-\lambda  \Pi_K^0 u (\bold{x}-\bold{M}_K)$,  this leads to
\begin{equation}\label{CRfinal01}
\parallel \RTsigS- \nabla_h \CRuPi\parallel_{0,\Om}^2= \frac{\lambda^2  }{144} \sum_{K\in \cT_h}H^2\parallel   u \parallel_{0,K}^2 + \cO(h^4|u|_{1, \Om}^2).
\end{equation}
It follows from the facts $\int_K \bold{x}-\bold{M}_K \dx=0$ and $\nabla_h \CRuPi-\nabla_h \CRu|_K\in P_0(K, \R^2)$ that
\begin{equation}\label{crIextra}
2(\RTsigS- \nabla_h \CRuPi, \nabla_h \CRuPi-\nabla_h \CRu)=-\lambda  (\Pi_h^0 u (\bold{x}-\bold{M}_K), \nabla_h \CRuPi-\nabla_h \CRu)=0.
\end{equation}
Recall the superconvergence \eqref{ECR1} of the RT element in Lemma \ref{Lm:bdRT}
\begin{equation*}
\parallel \PiRT\nabla u - \RTsigS\parallel_{0,\Om}\lesssim h^2 \big (| u|_{\frac{7}{2},\Om}+\kappa |\ln h|^{1/2}|u|_{2,\infty,\Om}\big ),
\end{equation*}
and the superclose property \eqref{CRESdiff} of the CR element in Lemma \ref{supereig1}
\begin{equation*}
\parallel\nabla_h \CRuPi - \nabla_h \CRu\parallel_{0,\Om}\lesssim h^2|u|_{2,\Om}.
\end{equation*}
Note that $\|u\|_{0, \Om}=1$. A substitution of \eqref{ECR1} , \eqref{CRESdiff}, \eqref{CRfinal01} and \eqref{crIextra} into \eqref{CRtotal0} leads to
\begin{equation*}
\begin{split}
\lambda-\CRlam&=\parallel (I-\PiRT)\nabla u \parallel_{0,\Om}^2   + \frac{\lambda^2  H^2}{144}+I_{\rm CR} +I_{\rm RT}+I_{\rm CR}^1+I_{\rm CR}^2+\cO(h^4|\ln h||u|_{{7\over 2},\Om}^2),
\end{split}
\end{equation*}
which completes the proof.
\end{proof}

According to Theorem \ref{Th:extraCR}, the asymptotic expansion of eigenvalues requires the analysis of   the following five terms
$$
\parallel (I-\PiRT)\nabla u \parallel_{0,\Om}^2,\ I_{\rm CR},\ I_{\rm RT},\ I_{\rm CR}^1,\ I_{\rm CR}^2.
$$
The fourth-order accurate expansions of the first two terms are analyzed  in Section \ref{sec:rt2} and \ref{sec:rtc}, respectively. Optimal estimates of $I_{\rm RT}$,  $I_{\rm CR}^1$ and $I_{\rm CR}^2$  are analyzed in Section \ref{sec:rth}, \ref{sec:cr1} and \ref{sec:cr2},  respectively.

\subsection{Taylor expansion of $\parallel (I-\PiRT)\nabla u \parallel_{0,\Om}^2$}\label{sec:rt2}

Define  two short-hand notations for the RT element
\begin{align*}
\phi_{\rm RT}^1(\bold{x})&=(x_1-M_1, -x_2+M_2)^T,\quad \phi_{\rm RT}^2(\bold{x})=(x_2-M_2, x_1-M_1)^T.
\end{align*}
Note that for $i=1$ and $2$,
$$
\rm{div} \phi_{\rm RT}^i = 0, \quad \nabla P_2(K, \R^2) = \rm{RT(K, \R^2)} + \text{span}\{ \phi_{\rm RT}^1,\ \phi_{\rm RT}^2\}.
$$
Define
\begin{equation}\label{RTcdef}
\RTbeta^{ij} := \frac{1}{h^2|K|}\int_K \big((I-\PiRT)\phi_{\rm RT}^i\big )^T(I-\PiRT)\phi_{\rm RT}^j\dx.
\end{equation}
\begin{lemma}\label{lm:gammaconstant}
Under Assumption \ref{ass:mesh}, constants $\RTbeta^{ij} $ in \eqref{RTcdef} are independent of the mesh size $h$.
\end{lemma}
\begin{proof}
Notice that
\begin{equation}\label{RTintEx}
\PiRT\phi_{\rm RT}^i= \sum_{j=1}^3 a_{\rm RT}^{ij} (\bold{x}-\bold{p}_j)\quad  \text{ with }\quad a_{\rm RT}^{ij}=\frac{1}{2|K|}\int_{e_j}(\phi_{\rm RT}^i)^T \bold{n}_j\ds.
\end{equation}
Since $\phi_{\rm RT}^i\in P_1(K, \R^2)$,
$$
a_{\rm RT}^{ij}=\frac{|e_j|}{2|K|} (\phi_{\rm RT}^i(\bold{m}_j))^T \bold{n}_j
$$
with $\bold{m}_j$ the midpoint of edge $e_j$.
This, together with Assumption \ref{ass:mesh}, implies that both $a_{\rm RT}^{1j}$ and $a_{\rm RT}^{2j}$ are constants independent of the mesh size $h$. It follows from \eqref{RTcdef} and \eqref{RTintEx} that
$$
\RTbeta^{ij} = \frac{1}{h^2|K|}\big ((\phi_{\rm RT}^i, \phi_{\rm RT}^j )_{0, K} - \sum_{k=1}^3 (a_{\rm RT}^{jk}\phi_{\rm RT}^i + a_{\rm RT}^{ik}\phi_{\rm RT}^j, \bold{x}-\bold{p}_k)_{0, K} + \sum_{k,l=1}^3 a_{\rm RT}^{ik}a_{\rm RT}^{jl} (\bold{x}-\bold{p}_k, \bold{x}-\bold{p}_l)_{0, K}\big ).
$$
By the definition of $\phi_{\rm RT}^i$ and Assumption \ref{ass:mesh}, for any $1\leq i,j\leq 2,\ 1\leq k,l\leq 3$,
$$
\frac{1}{h^2|K|}(\phi_{\rm RT}^i, \phi_{\rm RT}^j )_{0, K},\ \frac{1}{h^2|K|}(\phi_{\rm RT}^j, \bold{x}-\bold{p}_k)_{0, K},\  \frac{1}{h^2|K|}(\bold{x}-\bold{p}_k, \bold{x}-\bold{p}_l)_{0, K}
$$
are constants independent of $h$. Thus, constants
$
\RTbeta^{ij}
$
in \eqref{RTcdef}
are also independent of the mesh size $h$, which completes the proof.
\end{proof}

The following lemma presents the Taylor expansion of the interpolation error  of the RT element for any quadratic polynomials.
\begin{lemma}\label{RT2h40}
For any $w\in P_2(K, \R)$,
\begin{equation}\label{identity:RT2}
\begin{split}
\parallel (I-\PiRT)\nabla w\parallel_{0,K}^2=&h^2\big (\frac{\RTbeta^{11}}{4}\parallel\partial_{x_1x_1}w-\partial_{x_2x_2}w\parallel_{0,K}^2 +  \RTbeta^{22}\parallel\partial_{x_1x_2}w\parallel_{0,K}^2 \\
&+ \RTbeta^{12}\int_K(\partial_{x_1x_1}w-\partial_{x_2x_2}w)\partial_{x_1x_2}w\dx\big ),
\end{split}
\end{equation}
where $\RTbeta^{11}$, $\RTbeta^{12}$ and $\RTbeta^{22}$  in \eqref{RTcdef} are independent constant of the mesh size $h$.
\end{lemma}
\begin{proof}
For any $w\in P_2(K, \R)$, since
$
\nabla P_2(K, \R^2) = \rm{RT(K, \R^2)} + \text{span}\{ \phi_{\rm RT}^1,\ \phi_{\rm RT}^2\},
$
\begin{equation}\label{expan:rt0}
(I-\PiRT)\nabla w=a_1(I-\PiRT)\phi_{\rm RT}^1 + a_2(I-\PiRT)\phi_{\rm RT}^2
\end{equation}
with the coefficients $a_i$ to be determined. For any vector $v=(v_1, v_2)$, define
$$
D_1 v=\partial_{x_1}v_1 -\partial_{x_2}v_2,\quad D_2 v=\partial_{x_2}v_1 + \partial_{x_1}v_2.
$$
Note that
$
D_i\phi_{\rm RT}^j = 2\delta_{ij}
$
and
$
D_i\bold{x}=0
$
for any $1\leq i, j\leq 2$.
By applying the operators $D_1$ and $D_2$ to the both sides of \eqref{expan:rt0},
\begin{equation}\label{expan:rt1}
a_1={1\over 2}(\partial_{x_1x_1} - \partial_{x_2x_2}) w,\quad a_2=\partial_{x_1x_2} w.
\end{equation}
A substitution of \eqref{expan:rt1} to \eqref{expan:rt0} leads to
\begin{equation}\label{identity:RT}
(I-\PiRT)\nabla w=\frac{ (\partial_{x_1x_1}w-\partial_{x_2x_2}w)}{2}(I-\PiRT)\phi_{\rm RT}^1 +  \partial_{x_1x_2}w (I-\PiRT)\phi_{\rm RT}^2.
\end{equation}
It follows from \eqref{RTcdef} and \eqref{identity:RT} that
\begin{equation}\label{identity:RT20}
\begin{split}
\parallel (I-\PiRT)\nabla w\parallel_{0,K}^2=&h^2\big (\frac{\RTbeta^{11}}{4}\parallel\partial_{x_1x_1}w-\partial_{x_2x_2}w\parallel_{0,K}^2 +  \RTbeta^{22}\parallel\partial_{x_1x_2}w\parallel_{0,K}^2 \\
&+ \RTbeta^{12}\int_K(\partial_{x_1x_1}w-\partial_{x_2x_2}w)\partial_{x_1x_2}w\dx\big ),
\end{split}
\end{equation}
which completes the proof.
\end{proof}

For  a smooth enough function $u$, to obtain a fourth-order accurate expansion of the interpolation error $ (I-\PiRT)\nabla u$, an orthogonal property as analyzed in Lemma \ref{Lm:m} is needed. To this end, introduce the canonical interpolation of the Morley element in \cite{morley1968triangular,ming2006morley}: let $\Pi_M v|_K\in P_2(K, \R)$  with
\begin{equation*}
\PiM v(\bp_i)= v(\bp_i),\quad \mathlarger{\int}_{e_i} {\partial \PiM v\over \partial \bold{n}}\ds= \int_{e_i} {\partial v\over \partial \bold{n}}\ds\quad 1\leq i\leq 3.
\end{equation*}
It follows that
\begin{equation}\label{hessian:m}
\int_e \nabla (I-\PiM) v\ds=0,\quad
\int_K \nabla^2 (I-\PiM) v\dx=0 \text{ for any edge $e$ and any element $K$}.
\end{equation}
For any element $K$, there exists the following error estimates for the interpolation error \cite{ming2006morley}
\begin{equation}\label{interr:m}
|(I-\PiM) v|_{m, K}\lesssim h^{3-m} |v|_{3, K}, \ \forall\ 0 \leq m\leq 3.
\end{equation}

\begin{lemma}\label{Lm:m} Under the Assumption \ref{ass:mesh}, for any $u\in H^4(\Om, \R)$, it holds that
$$
\big|\sum\limits_{K\in\mathcal{T}_h}((I-\PiRT)\nabla \PiM u, (I-\PiRT)\nabla (I - \PiM)u)_{0, K}\big|\lesssim h^4|u|_{4, \Om}^2,
$$
\end{lemma}
\begin{proof}
By the definition of the interpolations $\PiRT$ and $\PiM$,
$$
\PiRT \nabla (I - \PiM)u=0.
$$
It remains to prove that
$$
\big|\sum\limits_{K\in\mathcal{T}_h}((I-\PiRT)\nabla \PiM u, \nabla (I - \PiM)u)_{0, K}\big|\lesssim h^4|u|_{4, \Om}^2.
$$
Let $\PiPt$ be the Larange interpolation operator of the $H^1$ conforming $P_3$ element, which leads to the following decomposition:
\be\label{orth:0}
\begin{split}
&\sum\limits_{K\in\mathcal{T}_h}((I-\PiRT)\nabla \PiM u, \nabla (I - \PiM)u)_{0,K}\\
=&\sum\limits_{K\in\mathcal{T}_h} ((I-\PiRT)\nabla \PiM \PiPt u, \nabla (I - \PiM)\PiPt u)_{0,K} \\
& +\sum\limits_{K\in\mathcal{T}_h} ((I-\PiRT)\nabla \PiM (I-\PiPt )u, \nabla (I - \PiM)\PiPt u)_{0,K}\\
&+\sum\limits_{K\in\mathcal{T}_h} ((I-\PiRT)\nabla \PiM u, \nabla (I - \PiM)(I-\PiPt) u)_{0,K}.
\end{split}
\ee
Since $| (I-\PiPt) u|_{i,\Om}\lesssim h^{4-i}|u|_{4, \Om}$ for $0\leq i\leq 3$, the second and third terms on the right--hand side of \eqref{orth:0}
 can be  estimated as,  respectively,
\be\label{orth:1}
\begin{split}
\big|\sum\limits_{K\in\mathcal{T}_h}(I-\PiRT)\nabla \PiM (I-\PiPt )u, \nabla (I - \PiM)\PiPt u)_{0,K} \big|&\lesssim h^4|u|_{4,\Om}^2,\\
\big|\sum\limits_{K\in\mathcal{T}_h}(I-\PiRT)\nabla \PiM u, \nabla (I - \PiM)(I-\PiPt) u)_{0,K} \big|&\lesssim h^4|u|_{4,\Om}^2.
\end{split}
\ee
To analyze the first term on the right-hand side of \eqref{orth:0}, let $v_h =\PiM \PiPt u$.
Since $\text{div}_h (\nabla_h v_h)$ is a piecewise constant, an integration by parts implies that
$$
\text{div} (I-\PiRT)\nabla v_h\big |_K={1\over |K|}\int_{\partial K} (I-\PiRT)(\nabla v_h) \cdot \bold{n}\ds=0
$$
for any element $K\in\mathcal{T}_h$. Consequently,
\begin{equation}\label{lm34n1}
((I-\PiRT)\nabla v_h, \nabla (I - \PiM)\PiPt u)_{0, K} = \sum_{i=1}^3 \int_{e_i} (I - \PiM)\PiPt u  ((I-\PiRT)\nabla v_h \cdot \bold{n})\ds.
\end{equation}
Given edge $e_i$ of element $K$, recall that $\bold{n}_i$ and  $\bold{t}_i$ are the unit outward normal vector and the unit tangent vector, respectively. Note that ${\partial^2v_h\over \partial \bold{t}_i\partial \bold{n}_i}\in P_0(K, \R)$ and
\begin{equation}\label{lm34n2}
(I-\PiRT)\nabla v_h\cdot \bold{n}_i|_{e_i}={\partial v_h\over \partial \bold{n}_i} - {\partial v_h\over \partial \bold{n}_i}(\bold{m}_i)=|e_i|{\partial^2 v_h\over \partial \bold{t}_i\partial \bold{n}_i} (\psi_{i-1}-{1\over 2}).
\end{equation}
with $\bold{m}_i$ the midpoint of edge $e_i$ and $\psi_i$ the barycenter coordinates.
It follows from \eqref{lm34n1} and \eqref{lm34n2} that
\begin{equation}\label{lm34n3}
((I-\PiRT)\nabla v_h, \nabla (I - \PiM)\PiPt u)_{0, K} = \sum_{i=1}^3  |e_i| {\partial^2 v_h\over \partial \bold{t}_i\partial \bold{n}_i}\int_{e_i} (I - \PiM)\PiPt u  (\psi_{i-1}-{1\over 2})\ds.
\end{equation}
Since $\PiM \PiPt u\in P_2(K, \R)$, the Taylor expansion indicates that
\begin{equation}\label{lm34n40}
\begin{split}
(I - \PiM)\PiPt u\big |_{e_i} = &(I - \PiM)\PiPt u(\bold{m}_i) + |e_i|{\partial (I - \PiM)\PiPt u\over \partial \bold{t}_i}(\bold{m}_i) (\psi_{i-1} - {1\over 2}) \\
&+  {|e_i|^2\over 2}{\partial^2 (I - \PiM)\PiPt u\over \partial \bold{t}_i^2}(\bold{m}_i) (\psi_{i-1} - {1\over 2})^2 \\
&+ {|e_i|^3\over 6}{\partial^3\PiPt u\over \partial \bold{t}_i^3}(\bold{m}_i) (\psi_{i-1} - {1\over 2})^3.
\end{split}
\end{equation}
The Taylor expansion of ${\partial (I - \PiM)\PiPt u\over \partial \bold{t}_i}$ reads
\begin{equation*}
\begin{split}
{\partial (I - \PiM)\PiPt u\over \partial \bold{t}_i}\big |_{e_i} = &{\partial (I - \PiM)\PiPt u\over \partial \bold{t}_i}(\bold{m}_i) + |e_i|{\partial^2 (I - \PiM)\PiPt u\over \partial \bold{t}_i^2}(\bold{m}_i) (\psi_{i-1} - {1\over 2})\\
&+ {|e_i|^2\over 2}{\partial^3 \PiPt u\over \partial \bold{t}_i^3}(\bold{m}_i) (\psi_{i-1} - {1\over 2})^2.
\end{split}
\end{equation*}
By \eqref{hessian:m} and the fact that $\int_{e_i} \psi_{i-1} - {1\over 2}\ds=0$ and  $\int_{e_i} (\psi_{i-1} - {1\over 2})^2\ds={|e_i|\over 12}$,
\begin{equation*}
0=\int_{e_i}{\partial (I - \PiM)\PiPt u\over \partial \bold{t}_i}\ds = |e_i|{\partial (I - \PiM)\PiPt u\over \partial \bold{t}_i}(\bold{m}_i) + {|e_i|^3\over 24}{\partial^3\PiPt u\over \partial \bold{t}_i^3}.
\end{equation*}
Thus,
\begin{equation}\label{lm34n4}
{\partial (I - \PiM)\PiPt u\over \partial \bold{t}_i}(\bold{m}_i) =- {|e_i|^2\over 24}{\partial^3  \PiPt u\over \partial \bold{t}_i^3}.
\end{equation}
Since $\int_{e_i} (\psi_{i-1}-{1\over 2})^j\ds=0$ for $j=1$ and $3$, a substitution of \eqref{lm34n4} into \eqref{lm34n40} yields
\begin{equation*}
\int_{e_i} (I - \PiM)\PiPt u  (\psi_{i-1}-{1\over 2})\ds= -{|e_i|^4\over 720}{\partial^3 \PiPt u\over \partial \bold{t}_i^3}.
\end{equation*}
By the above equation and \eqref{lm34n3},
\begin{equation*}
((I-\PiRT)\nabla v_h, \nabla (I - \PiM)\PiPt u)_{0, K} = -\sum_{i=1}^3 {|e_{i}|^4 \over 720}\int_{e_i}{\partial^2 v_h\over \partial \bold{t}_i\partial \bold{n}_i}  {\partial^3 \PiPt u  \over \partial \bold{t}_{i}^3} \ds.
\end{equation*}
Since $\PiPt u$ is continuous on interior edges, and $\bold{t}_i|_{K_1}=-\bold{t}_i|_{K_2}$ and $\bold{n}_i|_{K_1}=-\bold{n}_i|_{K_2}$, a summation over all the elements leads to
\begin{equation*}
\sum\limits_{K\in\mathcal{T}_h}((I-\PiRT)\nabla v_h, \nabla (I - \PiM)\PiPt u)_{0, K} = -\sum_{e\in\cE_h} {|e|^4 \over 720}\int_{e}\big [{\partial^2 (v_h - u)\over \partial \bold{t}_e\partial \bold{n}_e}\big ] {\partial^3 \PiPt u  \over \partial \bold{t}_e^3} \ds.
\end{equation*}
By the trace inequality, the triangle inequality and the above equation,
\begin{equation}  \label{orth:2}
\big |\sum\limits_{K\in\mathcal{T}_h}((I-\PiRT)\nabla v_h, \nabla (I - \PiM)\PiPt u)_{0, K}\big | \lesssim h^4|u|_{3,\Om}^2.
\end{equation}
A substitution of \eqref{orth:1} and \eqref{orth:2} into \eqref{orth:0} leads to
$$
\big |\sum\limits_{K\in\mathcal{T}_h}((I-\PiRT)\nabla \PiM u, \nabla (I - \PiM)u)_{0, K}\big |\lesssim h^4|u|_{4, \Om}^2,
$$
which completes the proof.
\end{proof}

Thanks to Lemmas \ref{RT2h40} and  \ref{Lm:m}, there exists the following fourth-order accurate expansion of the term $\parallel (I-\PiRT)\nabla u\parallel_{0, \Om}^2$ in Theorem \ref{Th:extraCR}.
\begin{lemma}\label{RT2h4}
For any  $u\in H^{4}(\Om,\mathbb{R})$,
\begin{equation} \label{identity:RTu}
\begin{split}
\parallel (I-\PiRT)\nabla u\parallel_{0,\Om}^2=& h^2\big (\frac{\RTbeta^{11}}{4}\parallel\partial_{x_1x_1}u-\partial_{x_2x_2}u\parallel_{0,\Om}^2 +  \RTbeta^{22}\parallel\partial_{x_1x_2}u\parallel_{0,\Om}^2 \\
&+ \RTbeta^{12}\int_{\Om}(\partial_{x_1x_1}u-\partial_{x_2x_2}u)\partial_{x_1x_2}u\dx\big ) + \cO(h^4|u|_{4, \Om}^2).
\end{split}
\end{equation}
\end{lemma}

\subsection{Refined analysis of $I_{\rm CR}=2((I-\PiRT)\nabla u   , \RTsigS-\nabla_h \CRuPi)-2\lambda(u-\PiCR u,  u)$}\label{sec:rtc} For this term, a direct use of the Cauchy-Schwarz inequality and the Taylor expansions of interpolation errors only leads to a suboptimal expansion since the second term of $I_{\rm CR}$ is essentially a consistency error and  only admits a third order convergence which can not be improved. The idea here is to  explore the relation \eqref{CRRT}
  $$
  2(\RTsigS- \nabla_h \CRuPi)|_K=-\lambda  \Pi_K^0 u (\bold{x}-\bold{M}_K)
  $$
between the RT element and the CR element for any $K\in\mathcal{T}_h$, and decompose the first term of $I_{\rm CR}$ into two terms: one cancels the second term of $I_{\rm CR}$ and the other one has an asymptotic expansion.  To this end, one needs the following crucial superconvergence of the
  inner product of the errors of the  canonical interpolation  of the CR element and the piecewise constant $L^2$ projection.

\begin{lemma}\label{lm:crapp}
If  two adjacent elements $K_1$ and $K_2$ form a parallelogram. For any $w\in P_2(K_1\cup K_2, \R)$ and $v\in P_1(K_1\cup K_2, \R)$, it holds that
\be\label{CRpi0}
(w-\PiCR w, v-\Pi_h^0 v)_{0, K_1\cup K_2}=0.
\ee
Furthermore, under the Assumption \ref{ass:mesh}, for any $u\in H^3(\Om, \R)$,
\begin{equation}\label{app:crid}
\big |(u-\PiCR u, u-\Pi_h^0 u)\big |\lesssim  h^4|u|_{3, \Om}^2.
\end{equation}
\end{lemma}
\begin{proof} In order to derive an expression of the error $w-\PiCR w$, define three basis functions
\begin{equation*}
\phi_{\rm CR}^i=(2\psi_{i-1}-1)(2\psi_{i+1}-1)-\frac{2}{3}\psi_i+\frac{1}{3}, \quad 1\leq i\leq 3,
\end{equation*}
with the barycentric coordinates $\{\psi_i\}_{i=1}^3$ of element $K$. These basis functions satisfy that
\begin{equation*}
\int_{e_j} \phi_{\rm CR}^i \ds=0,\qquad \forall 1\leq i,\ j\leq 3.
\end{equation*}
This implies that these functions are bubble functions of the canonical interpolation operator $\PiCR$ of the CR element. Thus, for any quadratic polynomial $w$,
\be\label{extemp}
(I-\PiCR)w=\sum_{i=1}^3 c_i\phi_{\rm CR}^i.
\ee
To compute the coefficients $c_i$, $i=1, 2, 3$, recall the gradient $\nabla \psi_i=-\frac{\bold{n}_i}{d_i}$ of the barycenter coordinate $\psi_i$ from  \eqref{nlambda}. This gives
\begin{equation*}
\frac{\partial^2}{\partial \bold{t}_j^2}\phi_{\rm CR}^i=-\frac{8}{|e_i|^2}\delta_{ij}.
\end{equation*}
Thus, taking second order tangential derivatives on both sides of \eqref{extemp} yields
\begin{equation*}
c_i=-{|e_i|^2\over 8|K|}\int_K \frac{\partial^2 w }{\partial \bold{t}_i^2}\dx,\quad 1\leq i\leq 3.
\end{equation*}
This and \eqref{extemp} lead to  an expression of the error $w-\PiCR w$ as follows
\be\label{CRint}
(I-\PiCR)w=-\frac{1}{8}\sum_{i=1}^3|e_i|^2\frac{\partial^2 w}{\partial \bold{t}_i^2}\phi_{\rm CR}^i.
\ee
For any linear polynomial $v$, it holds that
\be\label{Poneint}
(I-\Pi_K^0)v = \sum_{i=1}^3 v(\bp_i)(\psi_i - {1\over 3}), \quad
(\phi_{\rm CR}^i, \psi_j-{1\over 3})_{0, K}=
\left\{
\begin{array}{ll}
{4\over 135} |K|, &\text{if }i=j\\
{-2\over 135} |K|, &\text{if }i\neq j
\end{array}.
\right.
\ee
A combination of \eqref{CRint} and  \eqref{Poneint} leads to
\begin{equation}\label{lm31n1}
\begin{split}
(w-\PiCR w, v-\Pi_K^0 v)_{0, K}&=-\frac{1}{540}\sum_{i=1}^3\frac{\partial^2 w}{\partial \bold{t}_i^2}|e_i|^2|K|(2v(\bp_i) - v(\bp_{i-1}) - v(\bp_{i+1}))\\
&=-\frac{1}{270}\sum_{i=1}^3\frac{\partial^2 w}{\partial \bold{t}_i^2}|e_i|^2|K|\nabla v\cdot \overrightarrow{\bM_i\bp_i}.
\end{split}
\end{equation}
Suppose $\bp_{K_1}$ and $\bp_{K_2}$ are the vertices of the element $K_1$ and $K_2$, respectively, and the opposite edge $e_{K_1}$ to $\bp_{K_1}$ in $K_1$ is parallel to the opposite edge $e_{K_2}$ to $\bp_{K_2}$ in $K_2$. Let $\bM_{K_1}$ and $\bM_{K_2}$ be the midpoints of $e_{K_1}$ and $e_{K_2}$, respectively.  Since the elements $K_1$ and $K_2$ form a parallelogram,
\begin{equation}\label{lm31n2}
|K_1|=|K_2|, \quad \overrightarrow{\bM_{K_1}\bp_{K_1}}=-\overrightarrow{\bM_{K_2}\bp_{K_2}}.
\end{equation}
Since  $\nabla^2 w$ and $\nabla v$ are constant in $K_1\cup K_2$, a combination of \eqref{lm31n1} and \eqref{lm31n2} yields
\begin{equation*}
(w-\PiCR w, v-\Pi_h^0 v)_{0, K_1\cup K_2}=0,
\end{equation*}
which completes the proof for \eqref{CRpi0}.

The partition $\cT_h$ of domain $\Om$ includes the set of  parallelograms $\mathcal{N}_1$ and the set of a few remaining boundary triangles $\mathcal{N}_2$. Let $\kappa=|\mathcal{N}_2|$ denote the number of the elements in $\mathcal{N}_2$.
\begin{equation}\label{meshtotal}
(u-\PiCR u, u-\Pi_h^0 u)=\sum_{K\in \mathcal{N}_1} (u-\PiCR u, u-\Pi_h^0 u)_{0, K} + \sum_{K\in \mathcal{N}_2} (u-\PiCR u, u-\Pi_h^0 u)_{0, K}.
\end{equation}
A direct application of the Bramble-Hilbert Lemma to \eqref{CRpi0} leads  to
\begin{equation}\label{meshin}
\big |\sum_{K\in \mathcal{N}_1} (u-\PiCR u, u-\Pi_h^0 u)_{0, K}\big |\lesssim  h^4|u|_{3, \Om}^2.
\end{equation}
For any element $K\in\mathcal{N}_2$, it follows from the error estimates of \eqref{CR:est}, \eqref{Pi0est} and the triangle inequality that
\begin{equation*}
\big |(u-\PiCR u, u-\Pi_h^0 u)_{0, K}\big|\lesssim h^4|u|_{2, K}|u|_{1, \infty},
\end{equation*}
consequently,
\begin{equation}\label{meshbd}
\big |\sum_{K\in \mathcal{N}_2} (u-\PiCR u, u-\Pi_h^0 u)_{0, K}\big | \lesssim \sqrt{\kappa} h^4|u|_{3, \Om}^2.
\end{equation}
A substitution of \eqref{meshin} and \eqref{meshbd} into \eqref{meshtotal} leads to
\eqref{app:crid} and   completes the proof.
\end{proof}

The following lemma shows an asymptotic expansion of  $I_{\rm CR}$ in Theorem \ref{Th:extraCR}.
\begin{lemma}\label{lm:RTcross}
Suppose that $u\in H^3(\Om, \R)$. Under the Assumption \ref{ass:mesh},
$$
I_{\rm CR}=-{\lambda^2 H^2 \over 72}   + \cO(h^4|u|^2_{3, \Om}),
$$
\end{lemma}
\begin{proof}
Recall the definition of  $I_{\rm CR}$ in \eqref{crI}
$$
I_{\rm CR}=2((I-\PiRT)\nabla u   , \RTsigS-\nabla_h \CRuPi)-2\lambda(u-\PiCR u,  u).
$$
For the first term on the right-hand side, it follows from the relation \eqref{CRRT} between the CR element and the RT element that
\begin{equation*}
2((I-\PiRT)\nabla u   , \RTsigS-\nabla_h \CRuPi)= -\lambda\sum\limits_{K\in\mathcal{T}_h}\big ((\nabla u, \Pi_K^0 u (\bold{x-M}_K))_{0, K}-(\PiRT\nabla u, \Pi_K^0 u (\bold{x-M}_K))_{0,K}\big ).
\end{equation*}
One key for the analysis is to decompose the first term on  the right--hand side of the above equation. Indeed, since $\nabla_h \PiCR u|_K\in P_0(K,\R^2)$ and $\int_K \bold{x}-\bold{M}_K \dx=0$, it holds that
\begin{equation*}
(\nabla u, \Pi_K^0 u (\bold{x}-\bold{M}_K))_{0,K}=(\nabla_h (u-\PiCR u), \Pi_K^0 u (\bold{x}-\bold{M}_K))_{0,K}.
\end{equation*}
Thus, it follows from $(\bold{x}-\bold{M}_K)\cdot \bold{n}|_e\in P_0(e, \R)$ and an integration by parts that
\begin{equation*}
(\nabla u, \Pi_K^0 u (\bold{x}-\bold{M}_K))_{0,K}= -2(u-\PiCR u, u)+2 (u-\PiCR u,  u-\Pi_K^0u)_{0,K}.
\end{equation*}
Multiplying  both side of the above equation by $\lambda$ , there is a cancellation between the first term on the right--hand side of the above equation and the second term of $I_{\rm CR}$ while the second term on the right--hand side of the above equation can be bounded by the superconvergence analyzed in  \eqref{app:crid}.  This yields
\begin{equation}\label{ICRdef}
I_{\rm CR}= \lambda\sum\limits_{K\in\mathcal{T}_h}(\PiRT\nabla u, \Pi_K^0 u (\bold{x-M}_K))_{0,K}+\cO(h^4|u|_{3, \Om}^2).
\end{equation}
It remains to analyze the first term on  the right hand side of the above equation. For the gradient $\nabla u$, its  interpolation of the RT element  reads
$$
\PiRT\nabla u|_K =  \sum_{i=1}^3 {1\over 2|K|}\int_{e_i} {\partial u\over \partial n} \ds (\bold{x}-\bold{p}_i).
$$
Since $\int_K \bold{x}-\bold{M}_K \dx=0$,
$$
( \bold{x}-\bold{p}_i, \bold{x}-\bold{M}_K)=( \bold{x}-\bold{M}_K, \bold{x}-\bold{M}_K)={H^2|K|\over 36},\qquad \forall\ 1\leq i\leq 3.
$$
This implies that
\begin{equation}\label{lm37n1}
(\PiRT\nabla u , \Pi_K^0 u(\bold{x}-\bold{M}_K))_{0, K}=  { H^2 \over 72} \int_{\partial K}\Pi_K^0u {\partial u\over \partial n} \ds =  { H^2 \over 72} (\Pi_K^0u, \Delta u)_{0, K}.
\end{equation}
By the orthogonal property  of the constant $L^2$ projection operator $\Pi_K^0$,
\begin{equation*}
(\Pi_K^0u, \Delta u )_{0, K}=( u, \Delta u)_{0, K} +  (\Pi_K^0u-u, \Delta u -  \Pi_K^0\Delta u)_{0, K}.
\end{equation*}
Since $H^2=\cO(h^2)$, a combination of this, the error estimate \eqref{Pi0est} for the $L^2$ projection operator $\Pi_K^0$, and \eqref{lm37n1} yields
\begin{equation}\label{rtexp}
\sum\limits_{K\in\mathcal{T}_h}(\PiRT\nabla u , \Pi_K^0 u(\bold{x}-\bold{M}_K))_{0,K}
=\sum_{K\in \cT_h}{ H^2 \over 72}\int_K u \Delta u \dx + \cO(h^4|u|_{3, \Om}^2).
\end{equation}
Under the Assumption \ref{ass:mesh}, the constant $H^2$ is the same on all elements. Since $\Delta u=-\lambda u$ and $\|u\|_{0, \Om}=1$, a substitution  of \eqref{rtexp} into \eqref{ICRdef} leads to
$$
I_{\rm CR}= -{\lambda^2 H^2 \over 72}   + \cO(h^4|u|^2_{3, \Om}),
$$
which completes the proof.
\end{proof}
\subsection{Error estimate of $I_{\rm RT}=2(\nabla u-\PiRT\nabla u, \PiRT\nabla u-\RTsigS)$}\label{sec:rth}
The fact that $\PiRT\nabla u-\RTsigS$ is divergence free and admits the superconvergence \eqref{ECR1} leads to the following optimal analysis of $I_{\rm RT}$.
\begin{lemma}\label{RTh4}
Suppose that $(\RTsigS,\RTuS)$ is the RT solution of the discrete source problem \eqref{RTbdPro} and $u\in H^{\frac{7}{2}}(\Om,\mathbb{R})$. Under the Assumption \ref{ass:mesh},
$$
\big |I_{\rm RT} \big | \lesssim h^4|\ln h||u|_{\frac{7}{2},\Om}^2.
$$
\end{lemma}
\begin{proof}
With $\tau_h= \PiRT\nabla u - \RTsigS$, the superconvergence result of the RT element  \eqref{ECR1} leads to
$$
I_{\rm RT}  = (\nabla u - \RTsigS, \tau_h) + \|\tau_h\|_{0, \Om}^2=  (\nabla u, \tau_h) - (\RTsigS, \tau_h) + \cO(h^4|\ln h||u|_{\frac{7}{2},\Om}^2).
$$
It follows from the integration by parts and Problem \eqref{RTbdPro} that
$$
I_{\rm RT}  = \sum_{K\in\cT_h} (-\int_K u{\rm div} \tau_h \dx + \int_{\partial K} u \tau_h\cdot \bold{n} \ds) + (\RTuS, {\rm div} \tau_h) + \cO(h^4|\ln h||u|_{\frac{7}{2},\Om}^2).
$$
Since ${\rm div} \tau_h=0$, $\tau_h$ is H(div)-conforming, and $u$ vanishes on the boundary $\partial\Omega$, this completes the proof.
\end{proof}

\subsection{Error estimate  of $I_{\rm CR}^1=2(\nabla u -\PiRT \nabla u , \nabla_h \CRuPi-\nabla_h \CRu)$}\label{sec:cr1}
A direct combination of Cauchy-Schwarz inequality and the superclose property \eqref{CRESdiff} of the CR element only yields  a suboptimal analysis of $I_{\rm CR}^1$.  The idea here is to further make use of  the relation \eqref{CRRT} between the CR element and the RT element
 and decompose it into three terms: a vanishing term, a fourth order term, and a remaining term. By using the commuting property of $\PiCR$,  the discrete eigenvalue problem \eqref{discrete}  and the discrete source problem \eqref{CRbdPipro}, and fully  exploring the properties of the projection operator $\Pi_h^0$ and the uniformity  of the mesh, this remaining term can be in some sense transferred to a  consistency error with  respect to the  nonconforming function $\PiCR u - \CRuPi$. This, in fact, leads to the following superconvergence result.

\begin{lemma}\label{SupPih0}
Suppose Assumption \ref{ass:mesh} holds.  Let $s_h=\PiCR u - \CRuPi$, it holds that
  \begin{equation}
  \big|((I - \Pi_h^0) s_h,   (I - \Pi_h^0) \CRu)\big|\lesssim h^4|u|_{2, \Omega}^2,
  \end{equation}
  with $(\CRlam, \CRu)$ an eigenpair of  \eqref{discrete} by the CR element and $\CRuPi$ the solution of  \eqref{CRbdPipro}.
\end{lemma}
 \begin{proof}
 The direct use of the error estimate of the piecewise constant $L^2$ projection and the usual  Cauchy-Schwarz inequality   can only derive a second order convergence for this inner product.
 The idea herein is to fully  explore the properties of the projection operator $\Pi_h^0$ and the uniformity  of the mesh and transfer it to  a  consistency error. Indeed, given $K\in\mathcal{T}_h$, since $s_h,\ \CRu|_K\in P_1(K, \R)$, there exist the following expansions
$$
(I - \Pi_h^0) s_h|_K = \sum_{i=1}^2 {\partial s_h\over \partial x_i} (x_i - M_i), \quad  (I - \Pi_h^0) \CRu|_K = \sum_{i=1}^2 {\partial \CRu\over \partial x_i} (x_i - M_i),
$$
with $\bold{M}_K=(M_1, M_2)$ the centroid of element $K$. Let $a_{ij}= {1\over |K|}\int_K(x_i-M_i)(x_j-M_j)\dx$. Thus,
\begin{equation*}
((I - \Pi_h^0) s_h,   (I - \Pi_h^0) \CRu)_{0, K}=\sum_{i,j=1}^2  a_{ij}\int_K{\partial s_h\over \partial x_i} {\partial \CRu\over \partial x_j}\dx.
\end{equation*}
Since the mesh is uniform,  the constant $a_{ij}$ is equal for all the elements of the mesh and is of  $\cO(h^2)$. Together with the error estimate \eqref{CR:est} of the CR element, this gives
\begin{equation*}
\big |a_{ij}\sum_{K\in\cT_h}\int_K{\partial s_h\over \partial x_i}{\partial (\CRu - u)\over \partial x_j}\dx\big |\lesssim h^4|u|_{2, \Om}^2,
\end{equation*}
where the following error estimate is employed
\begin{equation}\label{lm39new}
\|s_h\|_{0, \Om} + h\|\nabla_h s_h\|_{0, \Om}\lesssim h^2|u|_{2, \Om}.
\end{equation}
Therefore,
\begin{equation*}
((I - \Pi_h^0) s_h,   (I - \Pi_h^0) \CRu)=\sum_{i,j=1}^2  a_{ij}\sum_{K\in\cT_h}\int_K{\partial s_h\over \partial x_i} {\partial  u\over \partial x_j}\dx +\cO(h^4|u|_{2, \Om}^2).
\end{equation*}
Here the term $\sum_{K\in\cT_h}\int_K{\partial s_h\over \partial x_i} {\partial  u\over \partial x_j}\dx$ is essentially a consistency error. In fact, an integration by parts gives
\begin{equation}\label{I1CR:1}
\sum_{K\in\cT_h}\int_K{\partial s_h\over \partial x_i} {\partial  u\over \partial x_j}\dx=-\sum_{e\in\cE^i}\int_e [s_h]  {\partial  u\over \partial x_j}n_i \ds +(s_h,  {\partial^2  u\over \partial x_i\partial x_j}).
\end{equation}
Since $s_h\in \VCR$ and $\int_e [s_h]\ds=0$ for any $e\in\cE^i$, the first term on the right--hand side of the above equation can be estimated as
\begin{equation*}
\big |\sum_{e\in\cE^i}\int_e [s_h]  {\partial  u\over \partial x_j}n_i \ds\big |= \big |\sum_{e\in\cE^i}\int_e [(I-\Pi_e^0)s_h]  (I-\Pi_e^0){\partial  u\over \partial x_j}n_i \ds\big |\lesssim h^2 |u|_{2, \Omega}^2.
\end{equation*}
By the error estimate of \eqref{lm39new}, the second term on the right--hand side of \eqref{I1CR:1} can be bounded as
$$
(s_h,  {\partial^2  u\over \partial x_i\partial x_j})\lesssim h^2|u|_{2, \Omega}^2.
$$
A summary of these estimates  completes the proof.
\end{proof}

\begin{lemma}\label{CRa1}
Let $(\CRlam, \CRu)$ be an eigenpair of  \eqref{discrete} by the CR element and $\CRuPi$ be the solution of  \eqref{CRbdPipro}. Suppose $u\in H^{\frac{7}{2}}(\Om,\mathbb{R})$. Under the Assumption \ref{ass:mesh},
$$
\big |I_{\rm CR}^1\big | \lesssim h^4|\ln h|^{1\over 2}|u|_{\frac{7}{2},\Om}^2.
$$
\end{lemma}
\begin{proof}By the relation \eqref{CRRT} between the CR element and the RT element,  there is the following decomposition:
\begin{equation}\label{lm39n0}
\begin{split}
I_{\rm CR}^1=&\ 2(\nabla u - \nabla_h \CRuPi , \nabla_h \CRuPi-\nabla_h \CRu) \\
&+ \lambda\sum\limits_{K\in\mathcal{T}_h} (\Pi_K^0 u (\bold{x}-\bold{M}_K) , \nabla_h \CRuPi-\nabla_h \CRu)_{0, K}\\
&+2(\RTsigS -\PiRT \nabla u, \nabla_h \CRuPi-\nabla_h \CRu) .
\end{split}
\end{equation}
Since $\int_K \bold{x-M}_K \dx=0$ and  $\nabla_h \CRuPi-\nabla_h \CRu|_K\in P_0(K, \R^2)$ for any element $K\in\mathcal{T}_h$, the second term on the right--hand side of the above equation  vanishes, namely,
\begin{equation*}
 \lambda\sum\limits_{K\in\mathcal{T}_h} (\Pi_K^0 u (\bold{x}-\bold{M}_K) , \nabla_h \CRuPi-\nabla_h \CRu)_{0, K}=0.
\end{equation*}
The third  term on the right-hand side of \eqref{lm39n0} can by bounded by the superconvergence result of the RT element in Lemma \ref{Lm:bdRT} and the superclose property \eqref{CRESdiff} of the CR element, which reads
\begin{equation*}
\big |(\RTsigS -\PiRT \nabla u, \nabla_h \CRuPi-\nabla_h \CRu)\big |\lesssim h^4|\ln h|^{1\over 2}|u|_{{7\over 2}, \Om}^2.
\end{equation*}
It  remains to analyze the first term on the right--hand side of \eqref{lm39n0}.
To this end, let $s_h=\PiCR u - \CRuPi$. Then, the commuting property of $\PiCR$ gives
$$
(\nabla_h (u - \CRuPi), \nabla_h \CRuPi-\nabla_h \CRu)=(\nabla_h s_h, \nabla_h \CRuPi-\nabla_h \CRu)
$$
It follows from the discrete eigenvalue problem \eqref{discrete}  and the discrete source problem \eqref{CRbdPipro} that
\begin{equation}\label{lm39n1}
(\nabla_h s_h, \nabla_h \CRuPi-\nabla_h \CRu)=(s_h, \lambda \Pi_h^0 u-\CRlam \CRu).
\end{equation}
The error estimate \eqref{CR:est} of the CR element implies
\begin{equation}\label{lm39n2}
\|\PiCR u - \CRuPi\|_{0, \Om} + \|\Pi_h^0 (u-\CRu)\|_{0, \Om} + |\lambda-\CRlam| \lesssim h^2|u|_{2, \Om}.
\end{equation}
Since
$$
\lambda \Pi_h^0 u-\CRlam \CRu = \lambda \Pi_h^0 (u-\CRu) + (\lambda-\CRlam)\Pi_h^0 \CRu +\CRlam (\Pi_h^0-I)\CRu,
$$
a combination of \eqref{lm39new}, \eqref{lm39n1} and \eqref{lm39n2} leads to
$$
(\nabla_h s_h, \nabla_h \CRuPi-\nabla_h \CRu)= \CRlam (s_h, (\Pi_h^0 -I) \CRu)+\cO(h^4|u|_{2, \Om}^2).
$$
It follows from \eqref{lm39n0}, the orthogonal property of  the piecewise constant $L^2$ projection operator $\Pi_h^0$ and the above equation that
\begin{equation*}
I_{\rm CR}^1=-2 \CRlam ((I - \Pi_h^0) s_h,   (I - \Pi_h^0) \CRu) +\cO(h^4|\ln h|^{1\over 2}|u|_{{7\over 2}, \Om}^2).
\end{equation*}
A substitution of Lemma \ref{SupPih0} into the above identity completes the proof.
\end{proof}

\subsection{Error estimate of $I_{\rm CR}^2=2(\PiRT\nabla u -\RTsigS, \RTsigS- \nabla_h \CRu)$}\label{sec:cr2}
The superconvergence of the RT element and the relation between the RT element and the CR element only lead to a suboptimal analysis of $I_{\rm CR}^2$.
The key idea for an optimal analysis of  $I_{\rm CR}^2$ is 
to exploit the $H({\rm{div}})$-conformity and the divergence-free  property of $\PiRT\nabla u -\RTsigS$ in $I_{\rm CR}^2$.

\begin{lemma}\label{CRa2}
Suppose that $(\RTsigS,\RTuS)$ is the solution of source problem \eqref{RTbdPro}, $(\CRlam, \CRu)$ is the corresponding eigenpair  of \eqref{discrete} by the CR element and $u\in H^{\frac{7}{2}}(\Om,\mathbb{R})$. Under the Assumption \ref{ass:mesh},
$$
|I_{\rm CR}^2|\lesssim h^4|\ln h|^{1\over 2} |u|_{\frac{7}{2},\Om}^2.
$$
\end{lemma}
\begin{proof}
Thanks to the relation \eqref{CRRT} between the CR element and the RT element, the term $I_{\rm CR}^2$ can be decomposed into the following two terms
\begin{equation}\label{CR2deco}
I_{\rm CR}^2=2(\PiRT\nabla u -\RTsigS,\nabla_h \CRuPi- \nabla_h \CRu)
- \lambda  (\PiRT\nabla u -\RTsigS, \Pi_h^0 u(\bold{x}-\bold{M}_K)).
\end{equation}
According to the superconvergence result in Lemma \ref{Lm:bdRT} and the superclose property  \eqref{CRESdiff} of the CR element,
\begin{equation}\label{RThigh}
\big |(\PiRT\nabla u -\RTsigS,\nabla_h \CRuPi- \nabla_h \CRu)\big |\lesssim h^4|\ln h|^{1\over 2} |u |_{\frac{7}{2},\Om}^2.
\end{equation}
To bound the second term on the right--hand side of \eqref{CR2deco}, let $\tau_h= \PiRT\nabla u -\RTsigS$ and $\phi_K ={1\over 2}((x_1-M_1)^2 + (x_2-M_2)^2)$ with the centroid $\bold{M}_K=(M_1, M_2)$.
Since $\bold{x}-\bold{M}_K=\nabla \phi_K$ and ${\rm div} \tau_h=0$, it follows from the integration by parts and the continuity of $\tau_h\cdot \bold{n}$ that
\begin{equation*}\label{RTelement}
(\PiRT\nabla u -\RTsigS, \Pi_h^0 u(\bold{x}-\bold{M}_K))=  \sum_{e\in \cE^i} \int_e \tau_h\cdot \bold{n}_e[\Pi_K^0u\phi_K]\ds.
\end{equation*}
Since $\phi_K=\cO(h^2)$, by the Cauchy-Schwarz inequality and the trace inequality,
$$
\big | \int_e (\tau_h\cdot \bold{n}_e) [\Pi_K^0u \phi_{K}]  \ds\big |\lesssim h^2\| \tau_h\cdot \bold{n}_e\|_{0, e}\| [\Pi_h^0 u]\|_{0, e}\lesssim h^2\|\tau_h\|_{0,\omega_e}|u|_{1,\omega_e}).
$$
Consequently, thanks to the superconvergence result in Lemma \ref{Lm:bdRT},
\begin{equation}\label{cr2i}
| (\PiRT\nabla u -\RTsigS, \Pi_h^0 u(\bold{x}-\bold{M}_K))| \lesssim h^2\|\tau_h\|_{0,\Omega}\|u\|_{1,\Omega}\lesssim h^4|\ln h|^{1\over 2} |u|_{\frac{7}{2},\Om}^2.
\end{equation}
A substitution of \eqref{RThigh} and \eqref{cr2i} into \eqref{CR2deco} yields
$$
\big |I_{\rm CR}^2 \big |\lesssim h^4|\ln h|^{1\over 2} |u|_{\frac{7}{2},\Om}^2,
$$
which completes the proof.
\end{proof}

The following asymptotic expansions of eigenvalues by the CR element come from the combination of Lemmas \ref{RT2h4}, \ref{lm:RTcross},  \ref{RTh4},   \ref{CRa1},  \ref{CRa2} and Theorem \ref{Th:extraCR}.
\begin{Th}\label{CR:extra}
Suppose that $(\lambda , u )$ is the eigenpair of \eqref{variance} with $u \in H^4(\Om,\mathbb{R})\cap H^1_0(\Om,\mathbb{R})$, and $(\CRlam, \CRu)$ is the corresponding  eigenpair of \eqref{discrete} by the CR element on an uniform triangulation $\cT_h$. It holds that
\begin{equation*}
\begin{split}
\lambda-\CRlam=& h^2\big (\frac{\RTbeta^{11}}{4}\parallel\partial_{x_1x_1}u-\partial_{x_2x_2}u\parallel_{0,\Om}^2 +  \RTbeta^{22}\parallel\partial_{x_1x_2}u\parallel_{0,\Om}^2 + \RTbeta^{12}\int_{\Om}(\partial_{x_1x_1}u-\partial_{x_2x_2}u)\partial_{x_1x_2}u\dx\big ) \\
&- \frac{\lambda^2  }{144}  H^2 +\cO(h^4|\ln h||u|_{4,\Om}^2).
\end{split}
\end{equation*}
\end{Th}

\subsection{Extrapolation eigenvalues}
Denote the approximate eigenvalues of the CR element on $\cT_h$ by $\CRlam^h$.
Suppose that eigenvalues of the CR element converge at a rate $\alpha$ with a fixed coefficient $C$, namely
\begin{equation}\label{exid}
\lambda-\CRlam^{h}=Ch^{\alpha} + \cO(h^{\beta}) \text{ with }\beta>\alpha.
\end{equation}
If the convergence rate $\alpha$ is known, define extrapolation eigenvalues
\begin{equation}\label{exmethod}
\lambda_{\rm CR, 1}^{\rm EXP}={2^{\alpha} \CRlam^{2h}-\CRlam^{h}\over 2^{\alpha}-1}.
\end{equation}
It is easy to verify that extrapolation eigenvalues
$
\lambda_{\rm CR, 1}^{\rm EXP}
$
in \eqref{exmethod} converge to eigenvalues at a higher rate $ \beta$. If eigenfunctions are smooth enough, say $u\in H^{4}(\Omega, \R)$, Theorem \ref{CR:extra}  indicates that on uniform traingulations,
$$
\big |\lambda - \lambda_{\rm CR, 1}^{\rm EXP}\big |\lesssim h^4|\ln h| |u |_{4, \Omega}^2.
$$
The extrapolation eigenvalues in \eqref{exmethod} converge at a higher rate 4.

If eigenfunctions are singular, the convergence rate $\alpha$ in \eqref{exid}  is unknown. Suppose that the higher order term $\cO(h^{\beta})$ is zero,
$$
\CRlam^{4h}-\lambda=4^{\alpha}C h^{\alpha}, \quad \CRlam^{2h}-\lambda=2^{\alpha}C h^{\alpha}, \quad \CRlam^{h}-\lambda= Ch^{\alpha}.
$$
Then,
$$
{\frac{\CRlam^{4h}-\CRlam^{2h}}{\CRlam^{2h}-\CRlam^{h}}=2^{\alpha}}.
$$
A substitution of the above relation into \eqref{exmethod} gives new extrapolation eigenvalues
\begin{equation}\label{exmethod2}
\lambda_{\rm CR, 2}^{\rm EXP}= \frac{\left(\CRlam^{4h}-\CRlam^{2h}\right) \CRlam^{h}- \left(\CRlam^{2h}-\CRlam^{h}\right) \CRlam^{2h}}{\CRlam^{4h}+\CRlam^{h}-2 \CRlam^{2h}}
\end{equation}
for  unknown convergence rate $\alpha$ in \eqref{exid}.

\section{Asymptotic expansions of eigenvalues by the ECR element}

Let $\ECRuSt \in \VECR$ be the solution of the following discrete source problem
\be\label{ECRbdpro2}
(\nabla \ECRuSt ,\nabla_h v_h)=(\lambda u ,v_h)\quad\text{ for any } v_h\in \VECR.
\ee
The equivalence between the ECR element and the RT element \cite{Hu2015The} is crucial for expansions of eigenvalues by the ECR element.
Thanks to \eqref{ECRRT} in Lemma \ref{lm:equiv}, a similar proof to the one in Lemma \ref{supereig1} leads to
\begin{equation}\label{ECR3}
\parallel \RTsigS - \nabla_h \ECRu\parallel_{0,\Om}\lesssim h^{2}|u|_{2,\Om},
\end{equation}
provided that $u\in  H^2(\Om, \mathbb{R})\cap H^1_0(\Om, \mathbb{R})$.

\begin{Lm}\label{Th:extraECR}
Suppose that $(\lambda, u)$ is the eigenpair of \eqref{variance} with $u \in H^{\frac{7}{2}}(\Om,\mathbb{R})\cap H^1_0(\Om,\mathbb{R})$, and $(\ECRlam, \ECRu)$ is the corresponding eigenpair of \eqref{discrete} by the ECR element. Under the Assumption \ref{ass:mesh},
\begin{equation}\label{ECRlambdaexpand}
\lambda-\ECRlam=
\parallel \nabla u-\PiRT\nabla u \parallel_{0,\Om}^2-2\lambda (u -\PiECR u , u-\Pi_h^0 u ) +I_{\rm ECR}
+ \cO(h^4|\ln h||u|_{\frac{7}{2},\Om}^2).
\end{equation}
with \begin{equation*}\label{Iecr}
I_{\rm ECR}=2(\nabla u-\PiRT\nabla u, \RTsigS-\nabla_h \ECRu).
\end{equation*}
\end{Lm}
\begin{proof}
A similar analysis to that for \eqref{CRlambdaexpand} in Theorem \ref{Th:extraCR} leads to
\begin{equation}\label{ECRexpF}
\lambda-\ECRlam=\parallel \nabla_h (u-\ECRu)\parallel_{0,\Om}^2-2\lambda(u-\PiECR u, u) + \cO(h^4|u|_{2,\Om}^2).
\end{equation}
With the interpolation $\PiRT\nabla u $ and the solution $\RTsigS $ of the discrete source problem \eqref{RTbdPro}, the first term on the right-hand side of \eqref{ECRexpF}  can be decomposed as
\begin{equation}\label{ECRtotal1}
\begin{split}
\parallel \nabla u-\nabla_h \ECRu\parallel_{0,\Om}^2=&\parallel \nabla u-\PiRT\nabla u \parallel_{0,\Om}^2 + \parallel \PiRT\nabla u-\RTsigS \parallel_{0,\Om}^2 + \parallel \RTsigS- \nabla_h \ECRu \parallel_{0,\Om}^2\\
&+2(\PiRT\nabla u-\RTsigS,\RTsigS- \nabla_h \ECRu)+I_{\rm RT}+ I_{\rm ECR},
\end{split}
\end{equation}
with $I_{\rm RT}=2(\nabla u-\PiRT\nabla u, \PiRT\nabla u-\RTsigS)$ defined in \eqref{crI}.
Since $\RTsigS$ is the RT element approximation of $\nabla u$, a combination of the superconvergence \eqref{ECR1} of the RT element, \eqref{ECR3} and the triangle inequality  bounds the
fourth term on the right--hand side of \eqref{ECRtotal1} as follows
\begin{equation*}
2\big |(\PiRT\nabla u-\RTsigS,\RTsigS- \nabla_h \ECRu)\big |\lesssim h^{4}|\ln h|^{1/2}|u|_{\frac{7}{2},\Om}|u|_{2,\Om}.
\end{equation*}
Since the second term on the right--hand side of \eqref{ECRtotal1} is anlyzed in the superconvergence result \eqref{ECR1} of the RT element and the fifth term $I_{\rm RT}$  is estimated in Lemma \ref{RTh4}, this and  \eqref{ECR3} give
\begin{equation*}
\parallel \nabla u-\nabla_h \ECRu\parallel_{0,\Om}^2=\parallel \nabla u-\PiRT\nabla u \parallel_{0,\Om}^2 +I_{\rm ECR} +  \cO(h^{4}|\ln h| |u|_{\frac{7}{2},\Om}^2).
\end{equation*}
By the definition of the interpolation $\PiECR$ in \eqref{ecrinterpolation},
\begin{equation*}
(u-\PiECR u, u)=(u-\PiECR u,u-\Pi_h^0 u).
\end{equation*}
A substitution of these two equations into \eqref{ECRexpF}  leads to
\begin{equation*}\label{ECRlambdaexpand}
\lambda-\ECRlam=
\parallel \nabla u-\PiRT\nabla u \parallel_{0,\Om}^2-2\lambda (u -\PiECR u , u-\Pi_h^0 u ) +I_{\rm ECR}
+ \cO(h^4|\ln h||u|_{\frac{7}{2},\Om}^2),
\end{equation*}
which completes the proof.
\end{proof}

For the ECR element, the following lemma analyzes a similar result to that in Lemma \ref{lm:crapp} for the CR element, namely, the superconvergence of the inner product of the canonical interpolations of the ECR element and the piecewise constant $L^2$ projection.
\begin{lemma}\label{lm:ecrapp}
If  two adjacent elements $K_1$ and $K_2$ form a parallelogram. For any $w\in P_2(K_1\cup K_2, \R)$ and $v\in P_1(K_1\cup K_2, \R)$ that
\be\label{ECRpi0}
(w-\PiECR w, v-\Pi_h^0 v)_{0, K_1\cup K_2}=0.
\ee
Furthermore, under the Assumption \ref{ass:mesh},
$$
\big |((I-\PiECR) u, (I-\Pi_h^0)u) \big |\lesssim h^4 |u|^2_{3, \Om},
$$
provided that $u\in H^3(\Om, \R)$.
\end{lemma}
\begin{proof} The main idea is to derive a refined  expansion of the error $\PiCR w -\PiECR w$ in terms of the second order tangential derivatives $\frac{\partial^2 w}{\partial \bold{t}_j^2 }$, where $\bold{t}_j$ are  the tangential vectors of the three edges $e_j$,  $j=1, 2, 3$,  of element $K$. To this end,  define $\phi_{\rm ECR}\in \VECR$ by
\begin{equation}\label{consDef}
\phi_{\rm ECR}(\bold{x})=2 - \frac{36}{H^2}\sum_{i=1}^2(x_i-M_i)^2
\end{equation}
with the centroid $\bold{M}_K=(M_1,M_2)$ and $H^2=\sum_{i=1}^3 |e_i|^2 $. It is easy to verify that
\begin{equation}\label{ECR3basis}
\int_{e_i}\phi_{\rm ECR} \ds=0,\ \forall 1\leq i\leq 3\quad \text{ and }\quad
\frac{1}{|K|}\int_K \phi_{\rm ECR} \dx=1.
\end{equation}
Note that the quadratic function $\phi_{\rm CR}^i$ from Lemma \ref{lm:crapp} satisfies
$
{1\over |K|}\int_K \phi_{\rm CR}^i \dx={1\over 9},
$
and the integral average of $\phi_{\rm CR}^i$ on any edge is zero. Thus,
$$
\PiECR\phi_{\rm CR}^i = {1\over 9}\phi_{\rm ECR}.
$$
It follows from the expansion of the interpolation error of the CR element in \eqref{extemp} that
$
w=\PiCR w -\sum_{i=1}^3{|e_i|^2\over 8|K|}\int_K \frac{\partial^2 w }{\partial \bold{t}_i^2}\dx\phi_{\rm CR}^i.
$
Thus,
\begin{equation*}
\PiECR w=\PiCR w - {1\over 9}\phi_{\rm ECR}\sum_{i=1}^3{|e_i|^2\over 8|K|}\int_K \frac{\partial^2 w }{\partial \bold{t}_i^2}\dx .
\end{equation*}
By the property of $\phi_{\rm ECR}$ in  \eqref{ECR3basis}, $\phi_{\rm ECR}=(I-\PiCR) \phi_{\rm ECR}$.  
This and the above identity yield
\begin{equation*}
\PiCR w -\PiECR w = {1\over 9}(I-\PiCR) \phi_{\rm ECR}\sum_{i=1}^3{|e_i|^2\over 8|K|}\int_K \frac{\partial^2 w }{\partial \bold{t}_i^2}\dx.
\end{equation*}
By the  expansion of the  inner product of the errors of the  canonical interpolation  of the CR element and the piecewise constant $L^2$ projection from  \eqref{lm31n1}, this leads to,
\begin{equation*}
\begin{split}
(\PiCR w -\PiECR w, v-\Pi_K^0 v)_{0, K}&={1\over 9}((I-\PiCR) \phi_{\rm ECR}, v-\Pi_K^0 v)_{0, K}\sum_{i=1}^3{|e_i|^2\over 8|K|}\int_K \frac{\partial^2 w }{\partial \bold{t}_i^2}\dx \\
&=-\frac{1}{2430}\sum_{j=1}^3\frac{\partial^2 \phi_{\rm ECR}}{\partial \bold{t}_j^2}|e_j|^2|K|\nabla v\cdot \overrightarrow{\bM_j\bp_j}\sum_{i=1}^3{|e_i|^2\over 8|K|}\int_K \frac{\partial^2 w }{\partial \bold{t}_i^2}\dx.
\end{split}
\end{equation*}
A direct calculation derives $
\frac{\partial^2 \phi_{\rm ECR}}{\partial \bold{t}_j^2}=-{72\over H^2},\ \forall 1\leq j\leq 3.
$
Thus,
\begin{equation*}
\begin{split}
(\PiCR w -\PiECR w, v-\Pi_h^0 v)_{0, K}&=\frac{4|K|}{135H^2}\big (\sum_{i=1}^3{|e_i|^2\over 8|K|}\int_K \frac{\partial^2 w }{\partial \bold{t}_i^2}\dx\big )\sum_{j=1}^3|e_j|^2\nabla v\cdot \overrightarrow{\bM_j\bp_j}.
\end{split}
\end{equation*}
For any $w\in P_2(K_1\cup K_2, \R)$ and $v\in P_1(K_1\cup K_2, \R)$,
Assumption \ref{ass:mesh} implies that the term $\frac{|K|}{H^2}\big (\sum_{i=1}^3{|e_i|^2\over 8|K|}\int_K \frac{\partial^2 w }{\partial \bold{t}_i^2}\dx\big )$ has the same value on elements $K_1$ and $K_2$, thus,
\begin{equation}\label{ecrpi0diff}
(\PiCR w -\PiECR w, v-\Pi_h^0 v)_{0, K_1\cup K_2}=0.
\end{equation}
A combination of Lemma \ref{lm:crapp} for the CR element and \eqref{ecrpi0diff} leads to
$$
(w -\PiECR w, v-\Pi_h^0 v)_{0, K_1\cup K_2}=0,
$$
which completes the proof for \eqref{ECRpi0}.

It follows from a similar proof for Lemma \ref{lm:crapp} and \eqref{ECRpi0} that
$$
\big |((I-\PiECR) u, u-\Pi_h^0 u) \big |\lesssim h^4 |u|^2_{3, \Om} ,
$$
which completes the proof.
\end{proof}
It remains to analyze the term $I_{\rm ECR}$. Thanks to  the superconvergence result \eqref{ECR1} for $\PiRT\nabla u-\sigma_{\rm RT}^{\lambda u}$,  and the equivalence \eqref{ECRRT} between the ECR element and the RT element, namely
$\sigma_{\rm RT}^{\lambda u}= \nabla_h \ECRuS $,
$$
I_{\rm ECR}=2(\nabla u-\PiRT\nabla u, \RTsigS-\nabla_h \ECRu)=2(\nabla u - \nabla_h \ECRuS , \nabla_h \ECRuS-\nabla_h \ECRu)   +\cO(h^4|\ln h|^{1\over 2}|u|_{\frac{7}{2}, \Om}^2).
$$
A similar proof to that in Lemma \ref{CRa1} for a similar term $I_{\rm CR}^1$ of the CR element leads to
\begin{equation}\label{lm42n1}
I_{\rm ECR}=-2((I - \Pi_h^0)s_h, \ECRlam (I - \Pi_h^0) \ECRu) +\cO(h^4|\ln h|^{1\over 2}|u|_{\frac{7}{2}, \Om}^2),
\end{equation}
with $s_h=\PiECR u - \ECRuS $.

\begin{lemma}\label{ECRa1}
Let $(\ECRlam, \ECRu)$ be the eigenpair of  \eqref{discrete} by the ECR element and $\ECRuS$ be the solution of  \eqref{ECRbdpro}. Suppose $u\in H^{\frac{7}{2}}(\Om,\mathbb{R}^2)$. Under the Assumption \ref{ass:mesh},
$$
\big |I_{\rm ECR}\big | \lesssim h^4|\ln h|^{1\over 2}|u|_{\frac{7}{2},\Om}^2.
$$
\end{lemma}
\begin{proof} 
It only needs to estimate the first term on the right--hand side of \eqref{lm42n1}. Since both $s_h$ and  $\ECRu$  are
  piecewise quadratic polynomials, the analysis in Lemma \ref{SupPih0} can not be directly employed. In fact, it follows from the orthogonal property of the piecewise constant $L^2$ projection that there is the following decomposition: 
 \begin{equation*}
\begin{split}
((I - \Pi_h^0)s_h,  (I - \Pi_h^0) \ECRu)
=((I - \Pi_h^0)s_h,  (I -\Pi_{\rm CR}) \ECRu)+((I - \Pi_h^0)s_h,  (I-\Pi_h^0) \Pi_{\rm CR} \ECRu).
\end{split}
\end{equation*}
Since  $\|\nabla_h s_h\|_{0, \Om}\lesssim h|u|_{2,\Om}$, the first term on the right--hand side of the above equation can be estimated as
$$
((I - \Pi_h^0)s_h,  (I -\Pi_{\rm CR}) \ECRu)\leq \|(I - \Pi_h^0)s_h\|_{0, \Omega}\|(I -\Pi_{\rm CR}) \ECRu\|\lesssim h^4|u|_{2, \Omega}^2.
$$
A substitution of this estimate into the above equation leads to
\begin{equation*}
\begin{split}
I_{\rm ECR}&=-2((I - \Pi_h^0)s_h, \ECRlam (I - \Pi_h^0)\Pi_{\rm CR} \ECRu) +\cO(h^4|\ln h|^{1\over 2}|u|_{\frac{7}{2}, \Om}^2)\\
&=-2((I - \Pi_h^0)(I-\Pi_{\rm CR})s_h, \ECRlam (I - \Pi_h^0)\Pi_{\rm CR} \ECRu) \\
 &\quad -2((I - \Pi_h^0)\Pi_{\rm CR}s_h, \ECRlam (I - \Pi_h^0)\Pi_{\rm CR} \ECRu)+\cO(h^4|\ln h|^{1\over 2}|u|_{\frac{7}{2}, \Om}^2).
\end{split}
\end{equation*}
A similar analysis to that in Lemma \ref{CRa1} for the CR element proves
$$
\big |-2((I - \Pi_h^0)\Pi_{\rm CR}s_h, \ECRlam (I - \Pi_h^0)\Pi_{\rm CR} \ECRu)\big | \lesssim h^4|\ln h|^{1\over 2}|u|_{\frac{7}{2},\Om}^2.
$$
Thus,
\begin{equation*}
\begin{split}
I_{\rm ECR}&=-2((I - \Pi_h^0)(I-\Pi_{\rm CR})s_h, \ECRlam (I - \Pi_h^0)\Pi_{\rm CR} \ECRu) +\cO(h^4|\ln h|^{1\over 2}|u|_{\frac{7}{2}, \Om}^2).
\end{split}
\end{equation*}
Given any element $K\in\mathcal{T}_h$,  recall $\phi_{\rm ECR}(\bold{x})$ defined in \eqref{consDef}:
\begin{equation*}
\phi_{\rm ECR}(\bold{x})=2 - \frac{36}{H^2}\sum_{i=1}^2(x_i-M_i)^2,
\end{equation*}
with $\bold{M}_K=(M_1, M_2)$  the centroid of element $K$ and $H^2=\sum\limits_{i=1}^3|e_i|^2$.
 Define
\begin{align}\label{compliment}
\phi_{\rm ECR}^1(\bold{x})&=(x_1-M_1)^2-(x_2-M_2)^2,
&\phi_{\rm ECR}^2(\bold{x})&=(x_1-M_1)(x_2-M_2).
\end{align}
Note that the functions $\phi_{\rm ECR}^1$, and $\phi_{\rm ECR}^2$  belong to the compliment space of the shape function space of the ECR element with respect to $P_2(K)$.  Since $\int_{e_i}\phi_{\rm ECR} \ds=0$ for edge $e_i$ of $K$, $i=1, 2, 3$, 
$$
(I-\Pi_{\rm CR})s_h|_K=\alpha_K \phi_{\rm ECR}(\bold{x}) \text{ with } \alpha_K=-\frac{H^2}{144}\Delta s_h|_K.
$$
Since there exists the expansion $(I - \Pi_K^0) \Pi_{\rm CR}\ECRu|_K  = \sum_{j=1}^2 (x_j -M_j)\partial_j \Pi_{\rm CR}\ECRu$ and $\int_K (x_j -M_j)\dx=0$, this leads to
\begin{equation*}
\begin{split}
I_{\rm ECR}&=-\frac{\ECRlam}{2}\sum\limits_{K\in\mathcal{T}_h}\sum_{i,j=1}^2\int_K\Delta s_h\partial_j \Pi_{\rm CR}\ECRu(x_i-M_i)^2(x_j-M_j)\dx+\cO(h^4|\ln h|^{1\over 2}|u|_{\frac{7}{2}, \Om}^2).
\end{split}
\end{equation*}
Define $a_{iij}^K=\frac{1}{|K|}\int_K(x_i-M_i)^2(x_j-M_j)\dx$. Note that $a_{iij}^K=\mathcal{O}(h^3)$, $\sum\limits_{K\in\mathcal{T}_h}\|\Delta s_h\|_{0,K}^2\lesssim |u|_{2, \Omega}^2$,  and $\|\nabla_h(u-\Pi_{\rm CR}\ECRu)\|_{0, \Omega}\lesssim h|u|_{2, \Omega}$. This leads to
\begin{equation*}
\begin{split}
I_{\rm ECR}&=-\frac{\ECRlam}{2}\sum\limits_{K\in\mathcal{T}_h}\sum_{i,j=1}^2a_{iij}^K\int_K\Delta s_h\partial_ju\dx+\cO(h^4|\ln h|^{1\over 2}|u|_{\frac{7}{2}, \Om}^2).
\end{split}
\end{equation*}
In addition, there holds that
$$
a_{iij}^{K_1}=-a_{iij}^{K_2}, i, j=1,2
$$
provided that $K_1$ and $K_2$ form a parallelogram. Recall that the partition $\cT_h$ of domain $\Om$ includes the set of parallelograms $\mathcal{N}_1$ and the set of a few remaining boundary triangles $\mathcal{N}_2$ with $\kappa=|\mathcal{N}_2|$  the number of the elements in $\mathcal{N}_2$.  For a parallelogram $K_1\cup K_2=:Q\in \mathcal{N}_1$, let $a_{iij}^Q=a_{iij}^{K_1}=-a_{iij}^{K_2}$. Then,
\begin{equation*}
\begin{split}
\sum\limits_{K\in\mathcal{T}_h}\sum_{i,j=1}^2a_{iij}^K\int_K\Delta s_h\partial_ju\dx
=&\sum\limits_{Q\in\mathcal{N}_1}\sum_{i,j=1}^2a_{iij}^Q\big(\int_{K_1}\Delta s_h\partial_ju\dx-\int_{K_2}\Delta s_h\partial_ju\dx\big)
\\
&+\sum\limits_{K\in\mathcal{N}_2}\sum_{i,j=1}^2a_{iij}^K\int_{K}\Delta s_h\partial_ju\dx.
\end{split}
\end{equation*}
Since the second term on the right-hand side of the above equation can be bounded as follows
$$
\sum\limits_{K\in\mathcal{N}_2}\int_{K}\Delta s_h\partial_ju\dx\leq \sum\limits_{K\in\mathcal{N}_2}\sqrt{|K|}\|\Delta s_h\|_{0,K}\|\partial_ju\|_{0,\infty}\lesssim \sqrt{\kappa}h\|u\|_{3,\Omega}^2, 
$$
it only needs to analyze the first term. To this end, introduce the $L^2$ projection operators
 $\Pi_Q^\ell: L^2(Q, \mathbb{R})\rightarrow P_\ell(Q, \mathbb{R})$, $\ell=1, 2$,  for $w\in L^2(Q, \mathbb{R})$
$$
 \int_Q \nabla^i \Pi_Q^\ell w \dx =\int_Q \nabla^i w\dx\quad \forall 0\leq i\leq \ell.
 $$
 This leads to the following decomposition:
  \begin{equation}\label{ecrK}
  \begin{split}
 &\int_{K_1}\Delta \Pi_{\rm ECR}u \partial_ju\dx-\int_{K_2}\Delta  \Pi_{\rm ECR}u \partial_ju\dx\\
 &=\int_{K_1}\Delta  \Pi_{\rm ECR}(I-\Pi_Q^2)u \partial_j(I-\Pi_Q^1)u\dx
 +\int_{K_1}\Delta  \Pi_{\rm ECR}(I-\Pi_Q^2)u \partial_j\Pi_Q^1u\dx\\
& \quad +\int_{K_1}\Delta  \Pi_Q^2u \partial_j(I-\Pi_Q^1)u\dx
 -\big(\int_{K_2}\Delta  \Pi_{\rm ECR}(I-\Pi_Q^2)u \partial_j(I-\Pi_Q^1)u\dx\\
&\quad +\int_{K_2}\Delta  \Pi_{\rm ECR}(I-\Pi_Q^2)u \partial_j\Pi_Q^1u\dx
 +\int_{K_2}\Delta  \Pi_Q^2u \partial_j(I-\Pi_Q^1)u\dx\big)\\
 &\quad +\int_{K_1}\Delta \Pi_{\rm ECR}\Pi_Q^2u \partial_j\Pi_Q^1u\dx-\int_{K_2}\Delta  \Pi_{\rm ECR}\Pi_Q^2 u \partial_j\Pi_Q^1u\dx
 \end{split}
 \end{equation}
 The first eight terms on the right--hand side of the above equation can be bounded by
 $ \mathcal{O}(h)|u|_{3, Q}^2$.  In order to analyze the last two terms, let $w=\Pi_Q^2u$. Then, for $K=K_1, K_2$, it holds that
$$
(I-\PiECR)w=\frac{ \|\partial_{ x_1x_1 } w-\partial_{ x_2x_2} w\|_{0, K}}{4|K|^{\frac{1}{2}}}(I-\PiECR)\phi_{\rm ECR}^1 +   \frac{\|\partial_{ x_1x_2}w\|_{0, K}}{|K|^{\frac{1}{2}}} (I-\PiECR)\phi_{\rm ECR}^2
$$
with  $\phi_{\rm ECR}^i$, $i=1, 2$, defined in \eqref{compliment}.  Apply the Laplacian operator to both sides of the above equation,
$$
\Delta \PiECR w= \Delta w + \frac{ \|\partial_{ x_1x_1 } w-\partial_{ x_2x_2} w\|_{0, K}}{4|K|^{\frac{1}{2}}}\Delta (\PiECR \phi_{\rm ECR}^1) +   \frac{\|\partial_{ x_1x_2}w\|_{0, K}}{|K|^{\frac{1}{2}}}  \Delta (\PiECR \phi_{\rm ECR}^2).
$$
Note that on both $K=K_1$ and $K=K_2$, $\Delta (\phi_{\rm ECR}^i)_K=0$, $i=1, 2$. Here $(\cdot)_K$ denotes the corresponding function for element $K$. In addition, a similar relation holds for the interpolation of the ECR element, namely,  $\Delta (\PiECR \phi_{\rm ECR}^i)_{K_1}=\Delta (\PiECR \phi_{\rm ECR}^i)_{K_2}$, $i=1, 2$. As a result, the last two terms on the right-hand side of \eqref{ecrK} cancel each other, and 
$$
\big |\int_{K_1}\Delta \Pi_{\rm ECR}u \partial_ju\dx-\int_{K_2}\Delta  \Pi_{\rm ECR}u \partial_ju\dx\big |\lesssim h|u|_{3, Q}^2.
$$
A similar argument gives
$$
\big |\int_{K_1} (\Pi_{K_1}^0 \Delta u) \partial_ju\dx-\int_{K_2} (\Pi_{K_2}^0 \Delta u) \partial_ju\dx\big |\lesssim h|u|_{3, Q}^2.
$$
A combination of  \eqref{RTbdPro}, \eqref{ECRRT} and $-\Delta u=\lambda u$  leads to
$$
\Delta s_h = \Delta \PiECR u - {\rm div}\RTsigS =\Delta \PiECR u - \Pi_h^0 \Delta u.
$$
Thus, a summary of these equations completes the proof.
\end{proof}

In the following theorem, asymptotic expansions of eigenvalues of the ECR element are established and employed to improve the accuracy of approximate eigenvalues  from second order to forth order by extrapolation methods.

Extrapolation eigenvalues of the ECR element follow from the similar definitions in \eqref{exmethod} and \eqref{exmethod2}.
Denote the approximate eigenvalues of the ECR element on $\cT_h$ by $\ECRlam^h$. If eigenfunctions are smooth enough, define extrapolation eigenvalues by
\begin{equation}\label{exmethod3}
\lambda_{\rm ECR, 1}^{\rm EXP}={2^{\alpha} \ECRlam^{2h}-\ECRlam^{h}\over 2^{\alpha}-1}
\end{equation}
with $\alpha=2$. If eigenfunctions are singular, define
\begin{equation}\label{exmethod4}
\lambda_{\rm ECR, 2}^{\rm EXP}= \frac{\left(\ECRlam^{4h}-\ECRlam^{2h}\right) \ECRlam^{h}- \left(\ECRlam^{2h}-\ECRlam^{h}\right) \ECRlam^{2h}}{\ECRlam^{4h}+\ECRlam^{h}-2 \ECRlam^{2h}}.
\end{equation}

\begin{Th}\label{ECR:extra}
Suppose that  $(\lambda, u)$ is the eigenpair of \eqref{variance} with $u \in H^4(\Om,\mathbb{R})\cap H^1_0(\Om,\mathbb{R})$, and $(\ECRlam, \ECRu)$ is the corresponding eigenpair of \eqref{discrete} by the ECR element. Under the Assumption \ref{ass:mesh},
\begin{equation}\label{ECRexpan}
\begin{split}
\lambda-\ECRlam=&h^2\big ( \frac{\RTbeta^{11}}{4}\parallel\partial_{x_1x_1}u-\partial_{x_2x_2}u\parallel_{0,\Om}^2 + \RTbeta^{12}\int_{\Om}(\partial_{x_1x_1}u-\partial_{x_2x_2}u)\partial_{x_1x_2}u\dx\\
&+  \RTbeta^{22}\parallel\partial_{x_1x_2}u\parallel_{0,\Om}^2\big )+\cO(h^4|\ln h||u |_{4,\Om}^2),
\end{split}
\end{equation}
with   constants $\{\RTbeta^{ij}\}_{i,j=1}^2$
defined in  \eqref{RTcdef}.

Therefore, extrapolation eigenvalues converge at a higher rate 4, namely,
$$
|\lambda_{\rm ECR, 1}^{\rm EXP} - \lambda|\lesssim h^4|\ln h||u |_{4,\Om}^2.
$$
\end{Th}

%

\section{Numerical examples}\label{sec:numerical}
This section presents three numerical tests. The first example computes eigenvalues of the Laplacian operator on a unit square,  the second one deals with eigenvalues of a second order elliptic operator with discontinuous coefficients, and the last one computes eigenvalues on a cracked domain.

\subsection{Example 1.}
In this example, the model problem \eqref{variance} on the unit square $\Om=(0,1)^2$ is considered. The exact eigenvalues are
$$\lambda =(m^2+n^2)\pi^2 ,\ m,\ n\ \text{are positive integers},$$
and the corresponding eigenfunctions are $u =2\sin (m\pi x_1)\sin (n\pi x_2)$. Since these eigenfunctions are smooth, the convergence rates of the eigenvalues are known to be $\alpha=2$. According to Theorem \ref{CR:extra} and \ref{ECR:extra}, the extrapolation eigenvalues $\lambda_{\rm CR, 1}^{\rm EXP}$ in \eqref{exmethod} and $\lambda_{\rm ECR, 1}^{\rm EXP}$ in \eqref{exmethod3} converge at same rate 4 on uniform triangulations.

\begin{figure}[!ht]
\setlength{\abovecaptionskip}{0pt}
\setlength{\belowcaptionskip}{0pt}
\centering
\includegraphics[width=10cm,height=8cm]{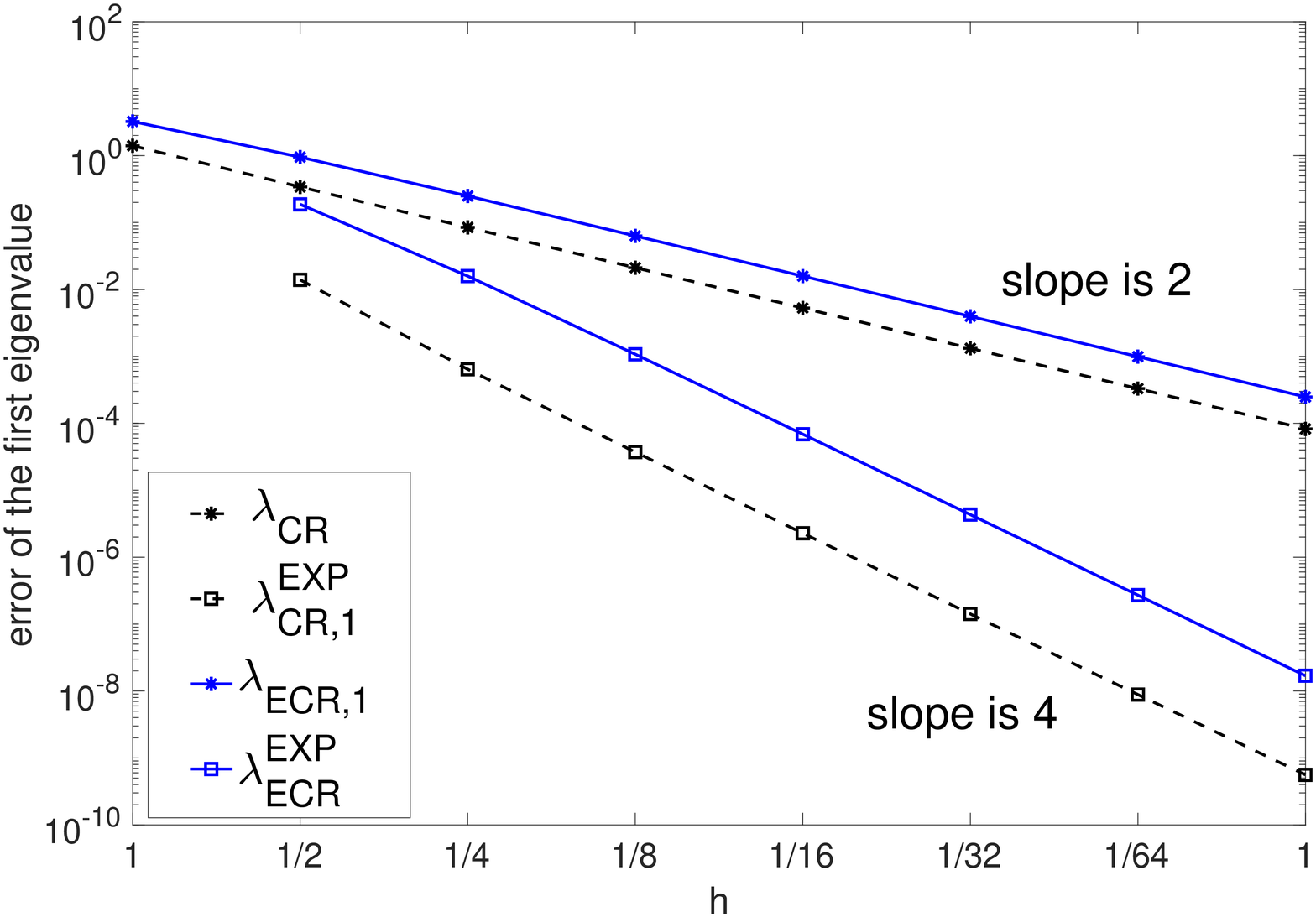}
\caption{\footnotesize{The errors of  the extrapolation eigenvalues on an uniform triangulation for Example 1.}}
\label{fig:squareExtrapolation}
\end{figure}


Figure \ref{fig:squareExtrapolation} plots the errors of the first approximate eigenvalues of the CR element, the ECR element and their corresponding extrapolation eigenvalues on  uniform triangulations. The initial triangulation $\cT_1$ consists of two right triangles, obtained by cutting the unit square with a north-east line. Each triangulation $\cT_i$ is refined into a half-sized triangulation uniformly, to get a higher level triangulation $\cT_{i+1}$. Figure \ref{fig:squareExtrapolation} verifies the optimal convergence rate 4  of the extrapolation eigenvalues $\lambda^{\rm EXP}_{\rm CR, 1}$ and $\lambda^{\rm EXP}_{\rm ECR, 1}$ in Theorem \ref{CR:extra} and \ref{ECR:extra}, respectively.

\begin{figure}[!ht]
\begin{center}
\begin{tikzpicture}[xscale=4,yscale=4]
\draw[-] (0,0) -- (0,1);
\draw[-] (0,0) -- (1,0);
\draw[-] (1,0) -- (1,1);
\draw[-] (0,1) -- (1,1);
\draw[-] (0,0.9) -- (0.9,1);
\draw[-] (0,0.9) -- (0.05,0);
\draw[-] (0.05,0) -- (0.9,1);
\draw[-] (1,0) -- (0.9,1);
\node[below, left] at (0,0) {(0,0)};
\node[below, right] at (1,0) {(1,0)};
\node[above, left] at (0,1) {(0,1)};
\node[above, right] at (1,1) {(1,1)};
\node[left] at (0,0.9) {(0,0.9)};
\node[below] at (0.05,0) {(0.05,0)};
\node[above] at (0.9,1) {(0.9,1)};
\end{tikzpicture}
\caption{\footnotesize A level one triangulation $\cT_1$ of $\Om$.}
\label{fig:nonuniformmesh}
\end{center}
\end{figure}
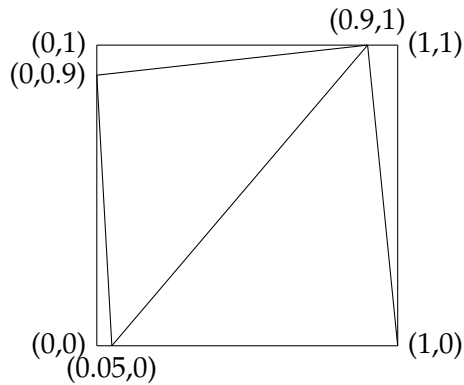

Consider the eigenvalue problem on other triangulations with the initial triangulation $\cT_1$ in Figure \ref{fig:nonuniformmesh}. Each triangulation $\cT_i$ is refined into a half-sized triangulation uniformly to get a higher level triangulation $\cT_{i+1}$. The errors of the corresponding eigenvalues $\CRlam $, $\ECRlam $, $\lambda^{\rm EXP}_{\rm CR, 1}$ and $\lambda^{\rm EXP}_{\rm ECR, 1}$  are recorded in Table \ref{tab:NonUniExtra}. Although these triangulations do not satisfy Assumption \ref{ass:mesh} anymore, most adjacent triangles form a nearly parallelogram. The percentage of this kind of triangles in $\cT_i$ increases as the level $i$ grows. Note that the superconvergence of  the extrapolation eigenvalues comes from the fact that some components of $\|\RTsigS-\PiRT \nabla u\|_{0, \Om}$ get canceled within nearly parallelograms. As the percentage of adjacent triangles forming nearly parallelograms increases, the accuracy of  the extrapolation eigenvalues increases and tends to the one on uniform triangulations. Table \ref{tab:NonUniExtra} shows that on such triangulations, which are not uniform anymore, the convergence rates of the extrapolation eigenvalues are still over 3.

\renewcommand\arraystretch{1.5}
\begin{table}[!ht]
\small
  \centering
    \begin{tabular}{c|ccccccc}
    \hline
          & $\cT_2$& $\cT_3$& $\cT_4$ & $\cT_5$& $\cT_6$&$\cT_7$&  $\cT_8$\\\hline
    $|\lambda -\CRlam |$    & 0.928068 & 2.22E-01 & 5.55E-02 & 1.39E-02 & 3.48E-03 & 8.69E-04 & 2.17E-04 \\
    rate  &   & 2.07  & 2.00  & 2.00  & 2.00  & 2.00  & 2.00 \\\hline
    $|\lambda -\lambda^{\rm EXP}_{\rm CR, 1}|$& 2.870925 & 1.39E-02 & 7.97E-05 & 3.45E-05 & 3.55E-06 & 3.04E-07 & 2.39E-08 \\
    rate  &       & 4.83  & 7.45  & 1.21  & 3.28  & 3.55  & 3.67  \\\hline
    $|\lambda -\ECRlam |$   & 2.683924 & 7.68E-01 & 2.01E-01 & 5.07E-02 & 1.27E-02 & 3.18E-03 & 7.96E-04 \\
    rate &  & 1.80  & 1.94  & 1.98  & 2.00  & 2.00  & 2.00 \\\hline
    $|\lambda -\lambda^{\rm EXP}_{\rm ECR, 1}|$      & 2.825196 & 1.30E-01 & 1.12E-02 & 7.78E-04 & 5.08E-05 & 3.27E-06 & 2.09E-07 \\
    rate  &       & 4.44  & 3.53  & 3.85  & 3.94  & 3.96  & 3.96  \\\hline
    \end{tabular}%
  \caption{\footnotesize The errors of  the extrapolation eigenvalues on nonuniform triangulations for Example 1.}
  \label{tab:NonUniExtra}%
\end{table}%

\subsection{Example 2}
This experiment considers the following eigenvalue problem
\begin{equation*}
\left\{
\begin{aligned}
 -{\rm div} (A \nabla u) \ &=\ \lambda  u& \  &\text{\quad in}\ \Om ,\\
 u &=0 &\  &\text{\quad on}\ \Gamma_1\cup \Gamma_2,\\
{\partial u\over \partial n} &=0 &\  &\text{\quad on}\ \Gamma_3,
\end{aligned}
\right.
\end{equation*}
where $A=2$ if $x_2<1$ and $A=1$ if $x_2>1$. The domain  is
$$
\Om=\big \{(x_1,x_2)\in \mathbb{R}^2:0< x_2< \sqrt{3}x_1, \sqrt{3}(1-x_1)<x_2\big \}
$$
with boundaries
\begin{equation*}
\begin{aligned}
\Gamma_1&=\big\{(x_1,x_2)\in \mathbb{R}^2: x_2=\sqrt{3}x_1,\ 0.5\le x_1\le 1 \big\},\\
\Gamma_2&=\big\{(x_1,x_2)\in \mathbb{R}^2: x_2=\sqrt{3}(1-x_1),\ 0.5\le x_1\le 1 \big\},\\
\Gamma_3&=\big\{(x_1,x_2)\in \mathbb{R}^2: x_1=1 ,\ 0\le x_2\le \sqrt{3} \big\}.
\end{aligned}\end{equation*}
The level one triangulation $\cT_1$ is obtained by refining the domain $\Om$ into four half-sized triangles. Each triangulation $\cT_i$ is refined into a half-sized triangulation uniformly, to get a higher level triangulation $\cT_{i+1}$.  Since  the exact eigenvalues of this problem are unknown, we take the first  eigenvalue by the conforming $\rm P_3$ element on the mesh $\cT_{9}$ as the reference eigenvalue.

\begin{figure}[!ht]
\setlength{\abovecaptionskip}{0pt}
\setlength{\belowcaptionskip}{0pt}
\centering
\includegraphics[width=9cm,height=8cm]{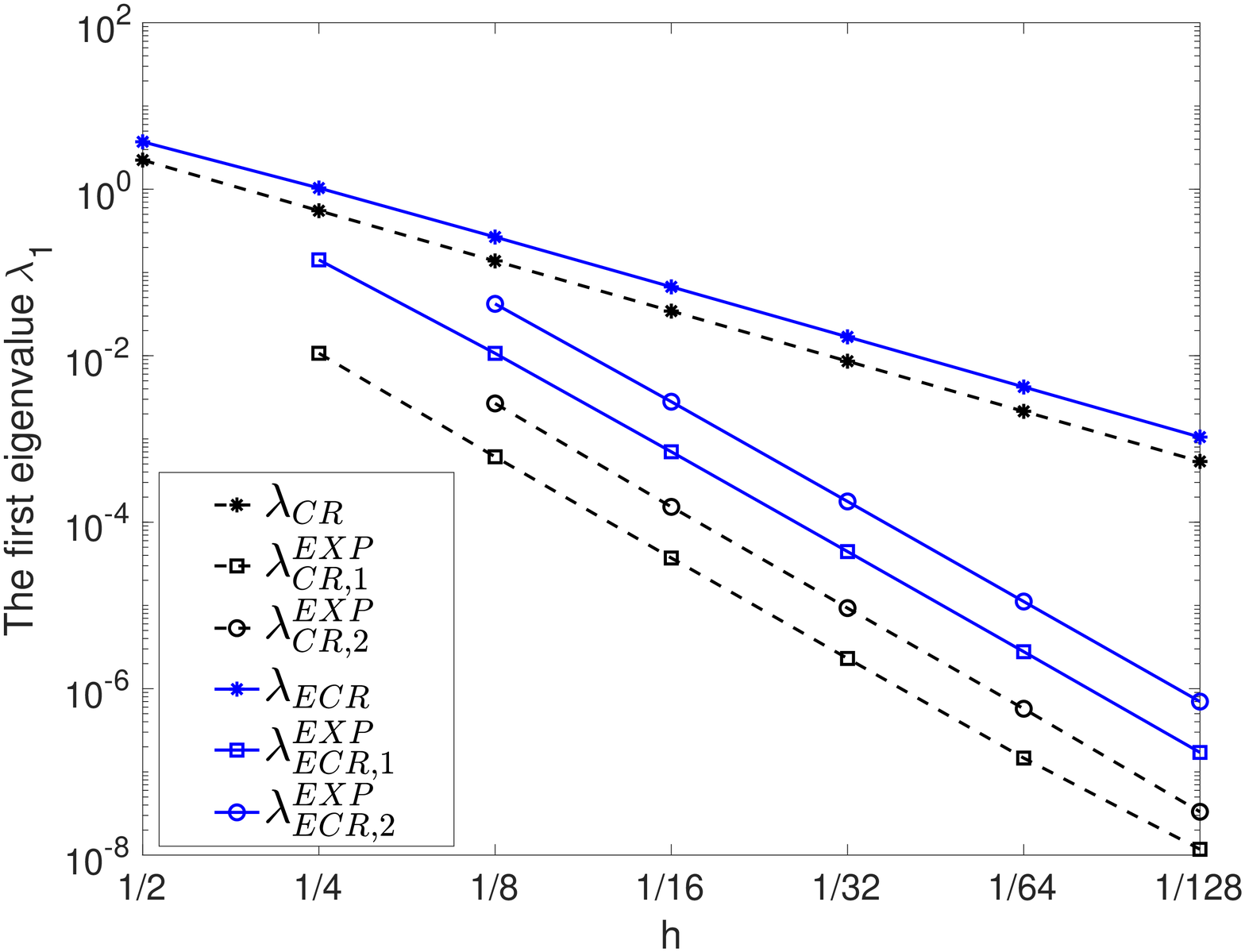}
\caption{\footnotesize{The errors of  the extrapolation eigenvalues for Example 2.}}
\label{fig:jumptri}
\end{figure}

Figure \ref{fig:jumptri} presents the errors of the first approximate eigenvalues of the CR element, the ECR element and their corresponding extrapolation eigenvalues on the uniform triangulations. As showed in Figure \ref{fig:jumptri}, the eigenvalues $\CRlam$ and $\ECRlam$ converge at the rate 2 and the extrapolation eigenvalues $\lambda^{\rm EXP}_{\rm CR, 1}$, $\lambda^{\rm EXP}_{\rm CR, 2}$, $\lambda^{\rm EXP}_{\rm ECR, 1}$ and $\lambda^{\rm EXP}_{\rm ECR, 2}$ converge at the same rate 4. This confirms the theoretical results in Theorem \ref{CR:extra} and \ref{ECR:extra}. Note that the errors of  the extrapolation eigenvalues $\lambda^{\rm EXP}_{\rm CR, 2}$ and $\lambda^{\rm EXP}_{\rm ECR, 2}$ are  slightly less than $\lambda^{\rm EXP}_{\rm CR, 1}$ and $\lambda^{\rm EXP}_{\rm ECR, 1}$, respectively.

\subsection{Example 3}
This example considers the model problem \eqref{variance} on a crack domain $\Om$ (see Figure \ref{fig:crack}).  A Dirichlet condition is applied on the boundary $\Gamma_D$ of the non-cracked domain $\overline{\Om}$. The crack is denoted by $\Gamma_C$ such that $\partial \Om=\Gamma_D \cup \Gamma_C$. The initial triangulation is shown in Figure \ref{fig:crack}, and  each triangulation is refined into a half-sized triangulation uniformly to get a higher level triangulation. Since  the exact eigenvalues of this problem are unknown, we take the first eight eigenvalues by the conforming $\rm P_3$ element on the mesh $\cT_{9}$ as reference eigenvalues. Since the convergence rates of eigenvalues are unknown, the extrapolation eigenvalues $\lambda_{\rm CR, 2}^{\rm EXP}$in \eqref{exmethod2}   and $\lambda_{\rm ECR, 2}^{\rm EXP}$ in \eqref{exmethod4} are considered for this example.

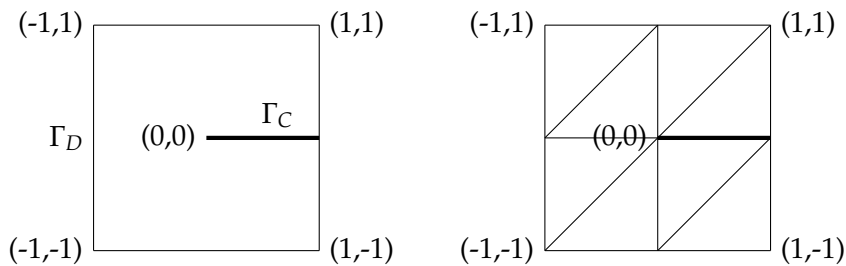
\begin{figure}[!ht]
\begin{center}
\begin{tikzpicture}[xscale=3,yscale=3]
\draw[-] (0,0) -- (0,1);
\draw[-] (0,0) -- (1,0);
\draw[-] (1,0) -- (1,1);
\draw[-] (0,1) -- (1,1);
\draw[ultra thick] (0.5,0.5) -- (1,0.5);
\node[below, left] at (0,0) {(-1,-1)};
\node[below, right] at (1,0) {(1,-1)};
\node[above, left] at (0,1) {(-1,1)};
\node[above, right] at (1,1) {(1,1)};
\node[below,left] at (0.5,0.5) {(0,0)};

\node[left] at (0,0.5) {$\Gamma_D$};
\node[above, right] at (0.7,0.6) {$\Gamma_C$};

\draw[-] (2,0) -- (2,1);
\draw[-] (2,0) -- (3,0);
\draw[-] (3,0) -- (3,1);
\draw[-] (2,1) -- (3,1);
\draw[-] (2,0.5) -- (2.5,0.5);
\draw[ultra thick] (2.5,0.5) -- (3,0.5);
\draw[-] (2.5,0) -- (2.5,1);
\draw[-] (2,0) -- (3,1);
\draw[-] (2,0.5) -- (2.5,1);
\draw[-] (2.5,0) -- (3,0.5);
\node[below, left] at (2,0) {(-1,-1)};
\node[below, right] at (3,0) {(1,-1)};
\node[above, left] at (2,1) {(-1,1)};
\node[above, right] at (3,1) {(1,1)};
\node[below,left] at (2.5,0.5) {(0,0)};
\end{tikzpicture}
\caption{\footnotesize The crack domain $\Om$ ( left ) and the initial triangulation $\cT_1$ ( right ) in Example 3.}
\label{fig:crack}
\end{center}
\end{figure}


For both the CR element and the ECR element, numerical experiments show that  the first and the sixth eigenvalues converge at the rate 1,  and  the other six eigenvalues converge at the optimal rate 2. Figure \ref{fig:crackerr} plots the errors of  the eigenvalues $\lambda_1$, $\lambda_2$, $\lambda_6$ and $\lambda_8$ by both nonconforming elements. For $\lambda_2$ and  $\lambda_8$, the corresponding eigenfunctions are smooth, and meets the regularity constraints of eigenfunctions in Theorem \ref{CR:extra} and \ref{ECR:extra}. As showed in Figure \ref{fig:crackerr}, the extrapolation eigenvalues $\lambda_{\rm CR, 2}^{\rm EXP}$  and $\lambda_{\rm ECR, 2}^{\rm EXP}$ improve the accuracy of  the eigenvalues from $\cO(h^2)$ to $\cO(h^4)$, which confirms the theoretical results. For $\lambda_1$ and $\lambda_6$, the corresponding eigenfunctions are not smooth and the approximate eigenvalues by both nonconforming elements converge at the rate 1. These two cases are not covered by the theory. For $\lambda_1$,  the extrapolation eigenvalues $\lambda_{\rm ECR, 2}^{\rm EXP}$ admit a higher accuracy $\cO(h^3)$, and  the eigenvalues $\lambda_{\rm CR, 2}^{\rm EXP}$ admit an even higher accuracy $\cO(h^4)$. Different from the case for $\lambda_1$, the convergence rate of  the extrapolation eigenvalues $\lambda_{\rm CR, 2}^{\rm EXP}$ and $\lambda_{\rm ECR, 2}^{\rm EXP}$ for $\lambda_6$ is slightly larger than 1. Despite the small improvement in the convergence rates for these two cases, the accuracy of  the eigenvalues is remarkably improved by  the extrapolation methods.

\begin{figure}[!ht]
\setlength{\abovecaptionskip}{0pt}
\setlength{\belowcaptionskip}{0pt}
\centering
\includegraphics[width=5cm,height=5cm]{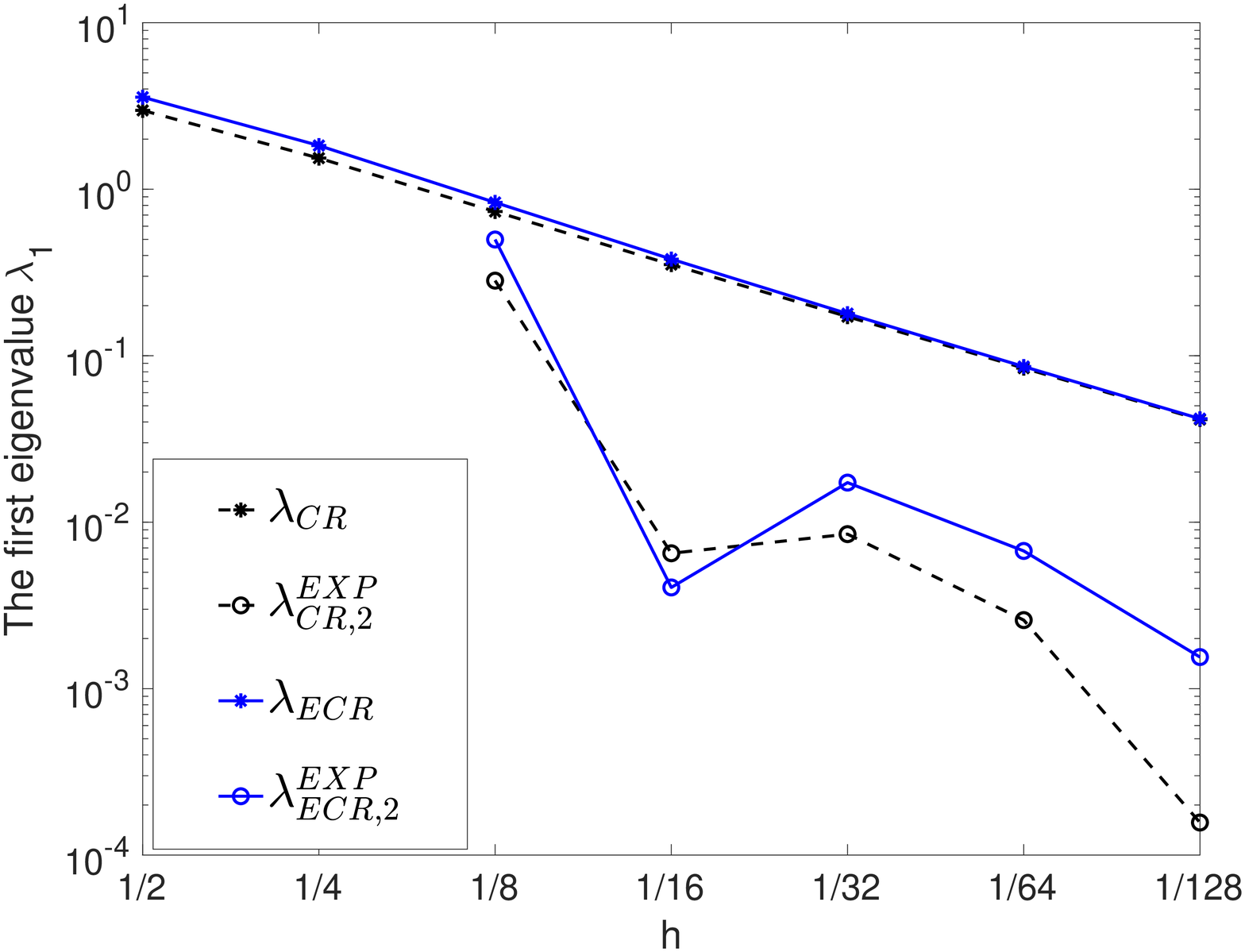}\hspace{.2in}
\includegraphics[width=5cm,height=5cm]{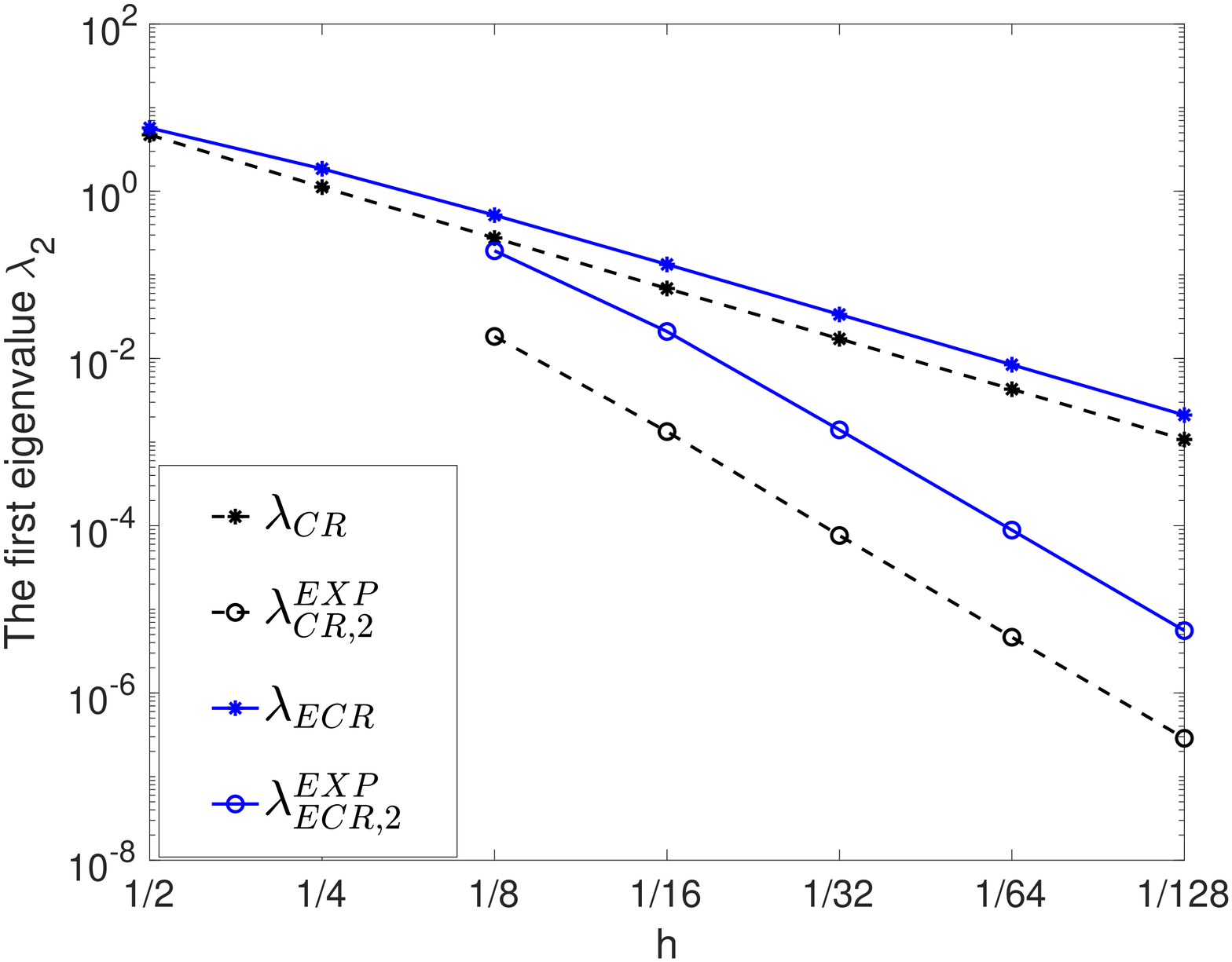}\vspace{.1in}
\includegraphics[width=5cm,height=5cm]{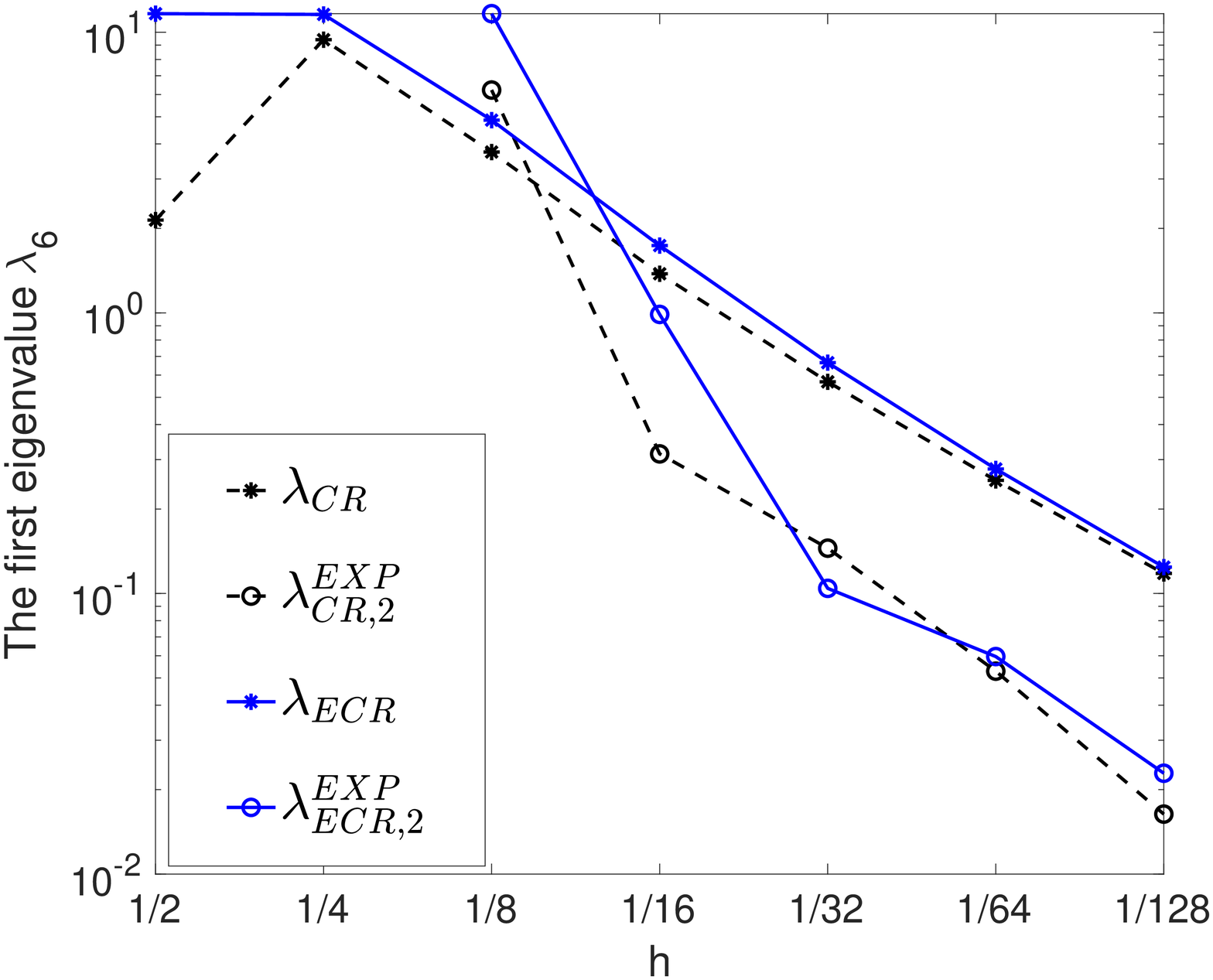}\hspace{.2in}
\includegraphics[width=5cm,height=5cm]{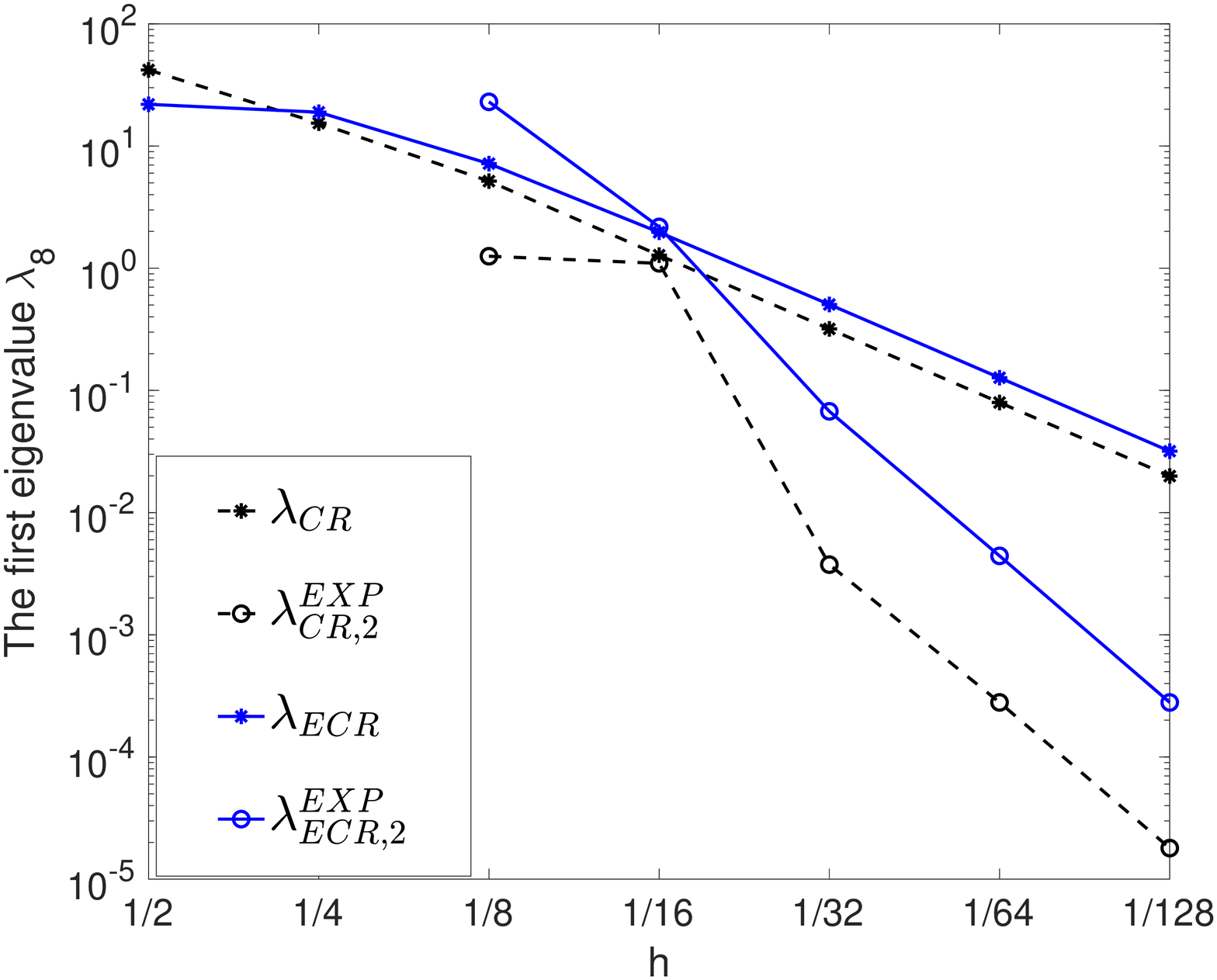}
\caption{\footnotesize{The errors of extrapolation eigenvalues for Example 3.}}
\label{fig:crackerr}
\end{figure}

%
%

\bibliographystyle{plain}
\bibliography{bibifile}

\end{document}